\documentclass[11pt,reqno]{amsart}
\usepackage[usenames]{color}
\usepackage{amsmath, amssymb, latexsym, multicol, amsthm, color, enumitem, comment, tikz, tikz-cd, mathrsfs}
\usepackage[all]{xy}
\usepackage{cancel}
\usepackage[T1]{fontenc}
\usepackage[greek,english]{babel}
\usepackage{hyperref}
\usepackage{cleveref}
\newtheoremstyle{natural}
{\parsep}   
{\parsep}   
{\normalfont}  
{0pt}       
{\bfseries} 
{.}         
{5pt plus 1pt minus 1pt} 
{}          

\newtheoremstyle{remark}
{\parsep}   
{\parsep}   
{\normalfont}  
{0pt}       
{\itshape} 
{.}         
{5pt plus 1pt minus 1pt} 
{}          
\theoremstyle{plain}
\newtheorem{thm}{Theorem}[section]

\newtheorem{lem}[thm]{Lemma}
\newtheorem{prp}[thm]{Proposition}
\newtheorem*{dfn}{Definition}

\theoremstyle{remark}
\newtheorem*{rmk}{Remark}

\numberwithin{equation}{section}

\crefname{thm}{theorem}{theorems}
\crefname{prp}{proposition}{propositions}
\crefname{cor}{corollary}{corollaries}

\newcommand{\Z}{\mathbb{Z}}
\newcommand{\Q}{\mathbb{Q}}
\newcommand{\N}{\mathbb{N}}
\newcommand{\R}{\mathbb{R}}
\newcommand{\C}{\mathbb{C}}
\newcommand{\F}{\mathbb{F}}

\newcommand{\A}{\mathbb{A}}

\newcommand{\bs}{\backslash}
\newcommand{\cb}[1]{\left\{{#1}\right\}}
\newcommand{\cbm}[2]{\left\{{#1}\;\middle|\;{#2}\right\}}
\newcommand{\Char}{\operatorname{char}}

\newcommand{\diag}{\operatorname{diag}}

\newcommand{\e}{\operatorname{e}}

\newcommand{\id}{\operatorname{id}}
\newcommand{\im}{\operatorname{Im}}

\newcommand{\ord}{\operatorname{ord}}
\newcommand{\p}{\mathbf{p}}
\newcommand{\pb}[1]{\left\langle{#1}\right\rangle}
\newcommand{\pd}[2]{\frac{\partial{#1}}{\partial{#2}}}
\newcommand{\ppmod}[1]{\hspace{-0.3cm}\pmod{#1}}

\newcommand{\ol}[1]{\overline{#1}}

\newcommand{\rank}{\operatorname{rank}}
\newcommand{\rb}[1]{\left({#1}\right)}

\newcommand{\sbe}{\subseteq}

\newcommand{\vb}[1]{\left| {#1} \right|}

\newcommand{\Kl}{\operatorname{Kl}}
\newcommand{\Klu}{\underline{\operatorname{Kl}}}

\newcommand{\GL}{\operatorname{GL}}
\newcommand{\SL}{\operatorname{SL}}

\newcommand{\Sp}{\operatorname{Sp}}

\newcommand{\mc}{\mathcal}
\newcommand{\mf}{\mathfrak}
\newcommand{\ms}{\mathscr}

\newcommand{\bdd}{\begin{center}\begin{tikzcd}}
\newcommand{\bd}{\begin{tikzcd}}
\newcommand{\edd}{\end{tikzcd}\end{center}}
\newcommand{\ed}{\end{tikzcd}}
\newcommand{\bdp}{\begin{center}\begin{tikzpicture}}
\newcommand{\edp}{\end{tikzpicture}\end{center}}
\newcommand{\bi}{\begin{itemize}}
\newcommand{\ei}{\end{itemize}}
\newcommand{\bt}{\begin{tikzpicture}}
\newcommand{\et}{\end{tikzpicture}}
\newcommand{\ba}{\[\begin{aligned}}
\newcommand{\ea}{\end{aligned}\]}
\newcommand{\bp}{\begin{pmatrix}}
\newcommand{\ep}{\end{pmatrix}}
\newcommand{\bv}{\begin{vmatrix}}
\newcommand{\ev}{\end{vmatrix}}
\newcommand{\bb}{\begin{bmatrix}}
\newcommand{\eb}{\end{bmatrix}}
\newcommand{\bB}{\begin{Bmatrix}}
\newcommand{\eB}{\end{Bmatrix}}
\newcommand{\bea}{\begin{enumerate}[label=(\alph*)]}
\newcommand{\ber}{\begin{enumerate}[label=(\roman*)]}
\newcommand{\ben}{\begin{enumerate}[label=(\arabic*)]}
\newcommand{\ee}{\end{enumerate}}

\title[Symplectic Kloosterman sums]{Symplectic Kloosterman sums and Poincaré series}

\author{Siu Hang Man}
\address{Siu Hang Man, Mathematisches Institut, Rheinische Friedrich-Wilhelms-Universität Bonn, Endenicher Allee 60, 53115 Bonn, Germany}
\email{shman@math.uni-bonn.de}
\date{}

\thanks{The author is supported in part by DAAD Graduate School Scholarship Programme.}

\makeatletter
\@namedef{subjclassname@2020}{%
  \textup{2020} Mathematics Subject Classification}
\makeatother
\subjclass[2020]{11L05, 11F30}
\keywords{Kloosterman sums, Poincaré series}

\begin{document}

\begin{abstract}
We prove power-saving bounds for general Kloosterman sums on $\Sp(4)$ associated to all Weyl elements via a stratification argument coupled with $p$-adic stationary phase methods. We relate these Kloosterman sums to the Fourier coefficients of $\Sp(4)$ Poincaré series.
\end{abstract}

\maketitle

\section{Introduction}
The classical Kloosterman sum is given by
\ba
S\rb{m, n; q} = \sum\limits_{\substack{x, y \in \Z/q\Z \\ xy \equiv 1\pmod{q}}}\e\rb{\frac{mx + ny}{q}}, 
\ea
where $\e\rb{x} = e^{2\pi i x}$. Kloosterman sums naturally appear in the Fourier expansion of $\GL(2)$ Poincaré series
\ba
P_m\rb{z; \nu} = \sum\limits_{\gamma \in \Gamma_\infty \bs \SL\rb{2,\Z}} \im\rb{\gamma z}^\nu\e\rb{m\rb{\gamma z}},
\ea
which play an important role in number theory. In \cite{BFG1988}, Bump, Friedberg and Goldfeld introduced $\GL(r)$ Poincaré series for $r\geq 2$, and gave a generalisation of Kloosterman sums to $\GL(3)$. The notion of Kloosterman sums was then generalised to $\GL(r)$ for $r\geq 2$ by Friedberg \cite{Friedberg1987}, and then to arbitrary simply connected Chevalley groups by Dąbrowski \cite{Dabrowski1993}. 

By methods of algebraic geometry, Weil \cite{Weil1948} obtained a bound for $\GL(2)$ Kloosterman sums
\ba
\vb{S\rb{m, n; q}} \ll \tau\rb{q} \rb{m, n, q}^{1/2} q^{1/2},
\ea
where $\tau$ denotes the divisor function. However, it remains a major open problem to give non-trivial bounds for Kloosterman sums in general, and currently only a small set of examples can be treated. Bounds for $\GL(3)$ Kloosterman sums were first obtained by Larsen \cite[Appendix]{BFG1988} and Stevens \cite{Stevens1987}, and were improved by Dąbrowski and Fisher \cite{DF1997}. Bounds for $\GL(4)$ Kloosterman sums were given by Huang \cite[Appendix]{GSW2019}. Friedberg \cite{Friedberg1987} generalised the results to $\GL(r)$ Kloosterman sums attached to certain Weyl elements. On reductive groups, Dąbrowski and Reeder \cite{DR1998} gave the size of Kloosterman sets, establishing a trivial bound for Kloosterman sums on reductive groups. 

Other than Poincaré series, another application of Kloosterman sums is found in the relative trace formula, which integrates an automorphic kernel over two subgroups with their respective characters. In particular, a prime application for bounds of Kloosterman sums is the analysis of the arithmetic side of the Petersson/Kuznetsov spectral summation formula. A more detailed description of this can be found in \cite{Blomer2019}.

Now we introduce the main results. Let 
\ba
G = \Sp(2r) &= \cbm{ M \in \GL(2r) }{ M^T JM = J}, & J &= \bp & I_n\\ -I_n\ep 
\ea
be the standard symplectic group. When $k$ is a field, $G(k)$ is the group of linear transformations of $k^{2r}$ preserving a symplectic bilinear form on $k^{2r}$. The standard torus and the standard unipotent subgroup of $G$ are given by
\ba
\scalebox{0.95}{$\displaystyle T = \cbm{\bp T_0\\ & T_0^{-1}\ep \in G}{T_0 \text{ diagonal}}, \quad U = \cbm{\bp U_0 & S\\ & (U_0^{-1})^T \ep\in G}{ U_0 \text{ upper triangular, unipotent}}$}
\ea
respectively. We denote by $N = N_G(T)$ the normaliser of $T$ in $G$. The Weyl group is given by $W := N_G(T)/T$. Let $w: N \to W$ be the canonical projection map with respect to this decomposition. For $n\in N$, we also define $U_n := U \cap n^{-1} U^T n$, and $\ol U_n := U \cap n^{-1} U n$. Note that $U_n, \ol U_n$ depend only on the image $w(n)$ of the canonical projection. 

Here we follow the notations in Stevens \cite{Stevens1987}. Let $p$ be a rational prime. We have a Bruhat decomposition
\ba
G\rb{\Q_p} = U\rb{\Q_p} N\rb{\Q_p} U_n\rb{\Q_p}.
\ea
For $n\in N\rb{\Q_p}$, we define 
\ba
C(n) &= U\rb{\Q_p} n U\rb{\Q_p} \cap G\rb{\Z_p},\\
X(n) &= U\rb{\Z_p} \bs C(n) / U_n\rb{\Z_p},
\ea
and projection maps
\ba
u: X(n) &\to U\rb{\Z_p} \bs U\rb{\Q_p},\\
u': X(n) &\to U\rb{\Q_p} / U_n\rb{\Z_p}
\ea
by the relation $x = u(x) n u'(x)$ for $x \in X(n)$. 

\begin{rmk}
In \cite{Stevens1987}, the notions above are defined for $\GL(r)$, but it is straightforward to check that the construction also works for $\Sp(2r)$, with essentially the same proofs. In particular, $X(n)$ is finite, and the projection maps $u,u'$ are well-defined. 
\end{rmk}

Let $n\in N\rb{\Q_p}$, $\psi_p$ a character of $U\rb{\Q_p}$ which is trivial on $U\rb{\Z_p}$, and $\psi'_p$ a character of $U_n\rb{\Q_p}$ trivial on $U_n\rb{\Z_p}$, such that $\psi'_p$ is the restriction of some character of $U\rb{\Q_p}$ trivial on $U\rb{\Z_p}$. Then the local Kloosterman sum is given by
\ba
\Kl_p\rb{n, \psi_p, \psi'_p} = \sum\limits_{x\in X(n)} \psi_p\rb{u(x)} \psi'_p\rb{u'(x)}.
\ea
If $\psi'_p$ is given as a character of $U\rb{\Q_p}$ which is trivial on $U\rb{\Z_p}$, we write $\Kl_p\rb{n, \psi_p, \psi'_p}$ to mean $\Kl_p\rb{n, \psi_p, \psi'_p|_{U_n\rb{\Q_p}}}$.

To define a global Kloosterman sum, let $n \in N(\Q)$, $\psi = \prod\limits_p \psi_p$ a character of $U(\A)$ which is trivial on $\prod\limits_p U\rb{\Z_p}$, and $\psi'$ a character of $U_n\rb{\A}$ trivial on $\prod\limits_p U_n\rb{\Z_p}$, such that $\psi'$ is the restriction of some character of $U\rb{\A}$ trivial on $\prod\limits_p U\rb{\Z_p}$. Then the global Kloosterman sum is given by
\ba
\Kl_p\rb{n, \psi, \psi'} = \prod\limits_p \Kl_p\rb{n, \psi_p, \psi'_p}.
\ea

\begin{rmk}
This definition of Kloosterman sums is different from the symplectic Kloosterman sums introduced by Kitaoka \cite{Kitaoka1984}, which are more relevant for classical $\Sp(4)$ Fourier expansions with respect to the upper right 2-by-2 block, which however is not a full parabolic subgroup. T\'oth \cite{Toth2013} proved some properties and estimates of such Kloosterman sums. The Kloosterman sums introduced here fit into the general framework of Kloosterman sums defined on reductive groups, see e.g. Dąbrowski \cite{Dabrowski1993}. 
\end{rmk}

For $G = \Sp\rb{4, \Q_p}$, a set of simple roots of $G$ with respect to the maximal torus $T$ is given by $\Delta = \cb{\alpha, \beta}$, where
\ba
\alpha\rb{\diag\rb{y_1, y_2, y_1^{-1}, y_2^{-1}}} &= y_1y_2^{-1}, & \beta\rb{\diag\rb{y_1, y_2, y_1^{-1}, y_2^{-1}}} &= y_2^2.
\ea
Then $\Psi^+ = \cb{\alpha, \beta, \alpha+\beta, 2\alpha+\beta}$ is a set of positive roots. We denote by $s_\alpha$ and $s_\beta$ the simple reflections in the hyperplane orthogonal to $\alpha$ and $\beta$ respectively. Then the Weyl group of $G$ with respect to $T$ is given by
\ba
W = \cb{1, s_\alpha, s_\beta, s_\alpha s_\beta, s_\beta s_\alpha, s_\alpha s_\beta s_\alpha, s_\beta s_\alpha s_\beta, s_\alpha s_\beta s_\alpha s_\beta}.
\ea
We fix once and for all representatives for $s_\alpha$ and $s_\beta$:
\ba
s_\alpha &= \bp &1\\-1\\&&&1\\&&-1\ep, & s_\beta &= \bp 1\\&&&1\\&&1\\&-1\ep. 
\ea
We also denote the long Weyl element $s_\alpha s_\beta s_\alpha s_\beta$ by $w_0$. Characters of $U\rb{\Q_p}$ trivial on $U\rb{\Z_p}$ are given by $\psi_{m_1, m_2}$ for $m_1, m_2\in\Z$, where
\begin{equation}\label{eq:psi4_def}
\psi_{m_1, m_2} \bp 1& x_1&*&*\\&1&*&x_2\\&&1\\&&-x_1&1\ep =\e\rb{m_1x_1+m_2x_2},
\end{equation}
where $\e:\Q_p/\Z_p \to\C^\times$ is the standard additive character satisfying $\e(p^{-r}) = e^{2\pi i/p^r}$. For $w\in W$, $r,s\in\Z$, we set
\begin{equation}\label{eq:nwrs}
n_{w, r, s} := \diag\rb{p^{-r}, p^{r-s}, p^r, p^{s-r}} w \in N\rb{\Q_p}.
\end{equation}
The exact conditions $r,s$ have to satisfy for $X(n_{w,r,s})$ to be nonempty are given in \Cref{section:Sp4Kloosterman}, but in general we require $r,s\geq 0$. By counting the number of terms in the Kloosterman sum \cite[Theorem 0.3]{DR1998}, we obtain a trivial bound
\ba
\vb{\Kl_p\rb{n_{w,r,s}, \psi, \psi'}} \leq p^{r+s}.
\ea
Now we state the main results of the paper. The first result concerns non-trivial bounds for local $\Sp(4)$ Kloosterman sums.
\begin{thm}\label{thm:local_bound}
Let $\psi = \psi_{m_1,m_2}$, $\psi' = \psi_{n_1,n_2}$ be characters of $U(\Q_p)/U(\Z_p)$, and $r,s$ be non-negative integers. Then we have
\begin{equation*}
\scalebox{0.92}{$
\begin{aligned}
\Kl_p(n_{\id,r,s}, \psi, \psi') &= 1 & &\text{if } r=s=0,\\
\vb{\Kl_p(n_{s_\alpha,r,s}, \psi, \psi')} &\ll p^{r/2}(m_1,n_1,p^r)^{1/2} & &\text{if } s=0,\\
\vb{\Kl_p(n_{s_\beta,r,s}, \psi, \psi')} &\ll p^{s/2}(m_2,n_2,p^s)^{1/2} & &\text{if } r=0,\\
\vb{\Kl_p\rb{n_{s_\alpha s_\beta, r, s}, \psi, \psi'}} &\ll \min\cb{p^{2s} \rb{m_1, p^{r-s}}, p^r \rb{m_2, p^s}^{1/2} \rb{n_2, p^s}^{1/2}} & &\text{if } s\leq r,\\
\vb{\Kl_p\rb{n_{s_\beta s_\alpha, r, s}, \psi, \psi'}} &\ll \min\cb{p^{3r} \rb{m_2, p^{s-2r}}, p^s \rb{m_1, n_1, p^r}} & &\text{if } 2r\leq s,\\
\vb{\Kl_p\rb{n_{s_\alpha s_\beta s_\alpha, r,s}, \psi, \psi'}} &\ll \begin{cases} 
p^{\frac{r}{3} + \frac{2s}{3} + \frac{2}{3}\min\cb{\ord_p(m_1)+s, \ord_p(n_1)+r} + \frac{1}{3}\ord_p(m_2)}\\
p^{r+\min\cb{\ord_p(m_2), r+\ord_p(n_1)}} + p^{r+\min\cb{\frac{s}{2}+\ord_p(m_1), r-\frac{s}{2}+\ord_p(n_1)}}\\
p^{r+\min\cb{\ord_p(m_2), r+\ord_p(n_1)}}\end{cases} & &\hspace{-0.18cm}\begin{array}{l}\text{if } s\leq r,\\ \text{if } r<s<2r,\\ \text{if } s=2r,\end{array}\\
\vb{\Kl_p\rb{n_{s_\beta s_\alpha s_\beta,r,s}, \psi, \psi'}} &\ll \begin{cases}
p^{\frac{s}{2} + \frac{r}{2} + \frac{1}{2} \ord_p(m_1) + \frac{1}{2}\min\cb{2r+\ord_p(m_2), s+\ord_p(n_2)}},\\
p^{s-\frac{r}{2}+\frac{1}{2}\ord_p(m_1)+\frac{1}{2}\min\cb{2r+\ord_p(m_2), s+\ord_p(n_2)}} \\
p^{s+\min\cb{\ord_p(m_1), \ord_p(n_2)}}\end{cases}. & &\hspace{-0.18cm}\begin{array}{l}\text{if } r\leq s/2,\\ \text{if } s/2<r<s,\\ \text{if } r=s,\end{array}\\
\vb{\Kl_p\rb{n_{w_0, r, s}, \psi, \psi'}} &\ll \min\cb{p^{\frac{1}{2}\ord_p(m_1m_2)},p^{ \frac{1}{2}\ord_p(n_1n_2)}} \rb{s+1} p^{\frac{r}{2} + \frac{3s}{4} + \frac{1}{2}\min\cb{r,s}}.
\end{aligned}$}
\end{equation*}
Moreover, the Kloosterman sum $\Kl_p(n_{w,r,s},\psi,\psi')$ vanishes if the conditions on the right are not satisfied.
\end{thm}

Now we state the bounds for global $\Sp(4)$ Kloosterman sums. Recall that characters of $U(\Q)/U(\Z)$ are given by $\psi_{m_1,m_2}$ for $m_1,m_2\in\Z$, where
\ba
\psi_{m_1,m_2}\bp 1& x_1 & *&*\\ &1&*&x_2\\&&1\\&&-x_1&1\ep = \e\rb{m_1x_1+m_2x_2}.
\ea
For $w\in W$, $c_1,c_2\in\N$, let
\ba
n_w(c_1,c_2) := \diag\rb{c_1^{-1},c_1c_2^{-1},c_1,c_1^{-1}c_2} w \in N(\Q).
\ea
\begin{thm}\label{thm:global_bound}
Let $\psi = \psi_{m_1,m_2}$, $\psi' = \psi_{n_1,n_2}$ be characters of $U(\Q)/U(\Z)$. For every $\varepsilon>0$ and $c_1,c_2\in \N$, the following bounds hold:
\ba
\Kl(n_{\id}(c_1,c_2),\psi,\psi') &= 1 & &\text{if } c_1=c_2=1,\\
\vb{\Kl(n_{s_\alpha}(c_1,c_2),\psi,\psi')} &\ll_\varepsilon (m_1,n_1,c_1)^{1/2} c_1^{1/2+\varepsilon} & &\text{if } c_2=1,\\
\vb{\Kl(n_{s_\beta}(c_1,c_2),\psi,\psi')} &\ll_\varepsilon (m_2,n_2,c_2)^{1/2} c_2^{1/2+\varepsilon} & &\text{if } c_1=1,\\
\vb{\Kl(n_{s_\alpha s_\beta}(c_1,c_2),\psi,\psi')} &\ll_\varepsilon \rb{c_2^2(m_1,c_1/c_2),c_1(m_2,c_2)^{1/2}(n_2,c_2)^{1/2}}(c_1c_2)^\varepsilon & &\text{if } c_2\mid c_1,\\
\vb{\Kl(n_{s_\beta s_\alpha}(c_1,c_2),\psi,\psi')} &\ll_\varepsilon \rb{c_1^3(m_2,c_2/c_1^2),c_2(m_1,n_1,c_1)}(c_1c_2)^\varepsilon & &\text{if } c_1^2\mid c_2,\\
\vb{\Kl(n_{s_\alpha s_\beta s_\alpha}(c_1,c_2),\psi,\psi')} &\ll_\varepsilon (m_1,n_1,c_1)(m_2,c_2)(c_1,c_2)(c_1c_2)^{1/3+\varepsilon} & &\text{if } c_2 \mid c_1^2,\\
\vb{\Kl(n_{s_\beta s_\alpha s_\beta}(c_1,c_2),\psi,\psi')} &\ll_\varepsilon (m_1,c_1)(m_2,n_2,c_2)(c_1^2,c_2) c_1^{-1/2} c_2^{1/2} (c_1c_2)^\varepsilon & &\text{if } c_1\mid c_2,\\
\vb{\Kl(n_{w_0}(c_1,c_2),\psi,\psi')} &\ll_\varepsilon (m_1m_2,n_1n_2,c_1c_2)^{1/2} (c_1,c_2)^{1/2} c_1^{1/2} c_2^{3/4} (c_1c_2)^\varepsilon. 
\ea
Moreover, the Kloosterman sum $\Kl(n_w(c_1,c_2),\psi,\psi')$ vanishes if the conditions on the right are not satisfied.
\end{thm}

\Cref{thm:local_bound,thm:global_bound} are proved in \Cref{section:Sp4Kloosterman_bound}. We develop a stratification of $\Sp(2r)$ Kloosterman sums in \Cref{section:stratification}, generalising the stratification of $\GL(r)$ Kloosterman sums introduced by Stevens \cite{Stevens1987}. Let 
\ba
\mc T := \cbm{\bp A\\ & cA^{-1} \ep \in \GL\rb{2r, \Z_p}}{A = \diag \rb{a_1, a_2, \cdots, a_r}, a_1, \cdots, a_r, c \in \Z_p^\times}.
\ea
be a subgroup of diagonal matrices. We will show that for $n\in N\rb{\Q_p}$, there is a group action $\mc T\times X(n) \to X(n)$ sending $(t, x)$ to $txs^{-1}$, where $s = n^{-1}tn$. The Kloosterman sum, as a sum over $X(n)$, can then be partitioned into sums over $\mc T$-orbits, in \Cref{thm:Stevens4.10}. The summands are then evaluated using results of Adolphson and Sperber \cite{AS1989} for multi-dimensional exponential sums of Laurent polynomials, as well as the $p$-adic stationary phase method for higher prime powers.

We give some brief comments on the bounds obtained here. As we shall see in \Cref{section:Sp4Kloosterman}, the Kloosterman sums corresponding $w=s_\alpha, s_\beta$ are just classical Kloosterman sums, and the bounds given here are just the optimal bounds for classical Kloosterman sums. The Kloosterman sums corresponding to $w=s_\alpha s_\beta, s_\beta s_\alpha$ can be expressed in terms of exponential sums of Laurent polynomials, whose optimal bounds are also well understood. Meanwhile, the bounds for $w=s_\alpha s_\beta s_\alpha, s_\beta s_\alpha s_\beta, w_0$, obtained using the stratification technique, are believed to be not optimal, since only the cancellations within individual $\mc T$-orbits are considered. In particular, we expect square-root cancellation for Kloosterman sums $\Kl_p(n_{w_0,r,s},\psi_{m_1,m_2},\psi_{n_1,n_2})$ when $m_1,m_2,n_1,n_2$ are coprime to $p$, but this is beyond reach of current methods.

Let $F: T\rb{\R^+} \to \C$ be a smooth function with rapid decay. Let $\psi, \psi'$ be characters of $U(\R)$ trivial on $U(\Z)$. For $g = uy\in G/K$, where $u \in U(\R)$, $y \in T\rb{\R^+}$, define $\mc F_\psi (g) := \psi\rb{\eta} F\rb{y}$. The symplectic Poincaré series associated to $F$ is given by
\ba
P_\psi (g) = \sum\limits_{\gamma\in P_0 \cap \Gamma \bs \Gamma} \mc F_\psi (\gamma g),
\ea
where $\Gamma = \Sp(2r, \Z)$, and $P_0$ is the standard minimal parabolic subgroup of $G$. The $\psi'$-th Fourier coefficient of $P_\psi(g)$ is given by
\ba
P_{\psi, \psi'} (g) = &\int_{U(\Z) \bs U(\R)} P_\psi\rb{ug}  \ol{\psi'} (u) du.
\ea

We compute in \Cref{section:sym_Poincare} the Fourier coefficients $P_{\psi, \psi'} (g)$ of the Poincaré series $P_\psi(g)$, in terms of auxiliary Kloosterman sums, which are also defined in \Cref{section:sym_Poincare}. The bounds given in \Cref{thm:local_bound} also apply to these auxiliary Kloosterman sums, via \Cref{prp:auxKl}.

\section*{Acknowledgement}
The author would like to thank Valentin Blomer for his guidance on the project.

\section{Stratification of symplectic Kloosterman sums} \label{section:stratification}

Consider the subgroup of diagonal matrices
\ba
\mc T := \cbm{\bp A\\ & cA^{-1} \ep \in \GL\rb{2r, \Z_p}}{A = \diag \rb{a_1, a_2, \cdots, a_r}, a_1, \cdots, a_r, c \in \Z_p^\times}.
\ea
Note that in general elements of $\mc T$ are not symplectic. 

\begin{lem}\label{lem:sympconj}
Let $u \in U\rb{\Q_p}$, and $t\in \mc T$. Then $tut^{-1}\in U\rb{\Q_p}$.
\end{lem}
\begin{proof}
Trivial.
\end{proof}

\begin{lem}\label{lem:torusconj}
Let $n\in N\rb{\Q_p}$, and $t \in \mc T$. Then $n^{-1} t n \in \mc T$.
\end{lem}
\begin{proof}
Suppose $t = \diag\rb{a_1, \cdots, a_r, ca_1^{-1}, \cdots, ca_r^{-1}}$. Then in general $n^{-1} t n$ has the form
\ba
n^{-1}tn = \diag\rb{\tau_{\sigma(1)} \rb{a_{\sigma(1)}}, \cdots, \tau_{\sigma(r)} \rb{a_{\sigma(r)}}, \tau_{\sigma(1)} \rb{ca_{\sigma(1)}^{-1}}, \cdots, \tau_{\sigma(r)} \rb{ca_{\sigma(r)}^{-1}}},
\ea
where $\sigma$ is a permutation of $\cb{1,\cdots, r}$, and $\tau_i: \cb{a_i, ca_i^{-1}} \to \cb{a_i, ca_i^{-1}}$ are permutations for $i = 1,\cdots, r$. Since $a_i = c\rb{ca_i^{-1}}^{-1}$, we see that 
\begin{equation*}
n^{-1}tn = \diag\rb{\tau_{\sigma(1)} \rb{a_{\sigma(1)}}, \cdots, \tau_{\sigma(r)} \rb{a_{\sigma(r)}}, c\rb{\tau_{\sigma(1)} \rb{a_{\sigma(1)}}}^{-1}, \cdots, c\rb{\tau_{\sigma(r)} \rb{a_{\sigma(r)}}}^{-1}} \in \mc T.\qedhere
\end{equation*}
\end{proof}

Let $x = unu' \in C(n)$, and $t \in \mc T$. By \Cref{lem:torusconj}, $s := n^{-1}tn \in \mc T$. By \Cref{lem:sympconj}, we see that 
\ba
t*x :=  txs^{-1} = \rb{tut^{-1}} n \rb{su's^{-1}} \in U\rb{\Q_p} n U\rb{\Q_p} \cap G\rb{\Z_p} = C(n). 
\ea
As conjugation by $t$ and $s$ preserves $U\rb{\Z_p}$ and $U_n\rb{\Z_p}$, this induces an action on $X(n)= U(\Z_p) \bs C(n) / U_n(\Z_p)$:
\ba
\mc T \times X(n) &\to X(n), & (t, x) &\mapsto t * x.
\ea

For characters $\psi: U\rb{\Q_p}/U\rb{\Z_p} \to \C^\times$, $\psi': U_n\rb{\Q_p}/U_n\rb{\Z_p} \to \C^\times$, decomposition of $X(n)$ into $\mc T$-orbits gives a decomposition of Kloosterman sums:
\ba
\Kl_p\rb{n, \psi, \psi} = \sum\limits_{x\in \mc T\bs X(n)} \sum\limits_{y \in \mc T * x} \psi\rb{u(y)} \psi' \rb{u'(y)}. 
\ea

Let $\alpha_i = e_i - e_{i+1}$, $1 \leq i \leq r-1$, and $\alpha_r = 2e_r$ be the simple roots of $T$ in $G$. Denote $\Delta = \cb{\alpha_1, \cdots, \alpha_r}$, and $\Delta_w = \cbm{\alpha \in \Delta}{w(\alpha)<0}$. Let $U_{\alpha_i}(\Q_p) \sbe U(\Q_p)$ be the root subgroup corresponding to $\alpha_i$. For $u\in U(\Q_p)$, we also use $\alpha_i$ to denote the canonical projection map
\ba
\alpha_i: U(\Q_p) \to U_{\alpha_i}(\Q_p) \simeq \Q_p.
\ea
Explicitly, for $u = (u_{ij}) \in U(\Q_p)$, the projection maps are given by
\ba
\alpha_i(u) &= u_{i,i+1}, & &1\leq i\leq r-1,\\
\alpha_r(u) &= u_{r,2r}.
\ea

Characters of $U\rb{\Q_p}/U\rb{\Z_p}$ have the form
\ba
\psi(u) = \psi_{n_1,\cdots, n_r}(u) &:= \prod\limits_{i=1}^r\e\rb{n_i \alpha_i(u)}, & &n_i\in\Z,
\ea
where $\e: \Q_p/\Z_p\to\C^\times$ is the standard additive character. For $x = u(x) n u'(x)$, define projections
\ba
\kappa_i (x) &:=\alpha_i(u(x)), & \kappa'_i (x) &:= \alpha_i(u'(x)), & 1\leq i\leq r.
\ea
Note that $\kappa'_i(x) = 0$ unless $\alpha_i \in \Delta_w(n)$. For $u\in U(\Q_p)$, and $t = \diag\rb{a_1, \cdots, a_r, ca_1^{-1}, \cdots, ca_r^{-1}} \in \mc T$, we compute that
\ba
\alpha_i(tut^{-1}) &= a_ia_{i+1}^{-1} \alpha_i(u), & &1\leq i\leq r-1,\\
\alpha_r(tut^{-1}) &= c^{-1}a_r^2 \alpha_r(u).
\ea
Suppose $t = \diag(a_1, \cdots, a_r, ca_1^{-1}, \cdots, ca_r^{-1}) \in \mc T$, and $s = n^{-1}tn = \diag(a'_1, \cdots, a'_r, {ca'_1}^{-1}, \cdots {ca'_r}^{-1}) \in \mc T$. Note from the proof of \Cref{lem:torusconj} that we have the same $c$. Then we have
\begin{equation}\label{eq:kappa_i}
\begin{split}
\kappa_i \rb{t*x} &= a_ia_{i+1}^{-1} \kappa_i(x), \quad 1\leq i\leq r-1,\\
\kappa_r \rb{t*x} &= c^{-1}a_r^2 \kappa_r(x), 
\end{split}
\end{equation}
and
\begin{equation}\label{eq:kappap_i}
\begin{split}
\kappa'_i \rb{t*x} &= a'_i {a'_{i+1}}^{-1} \kappa'_i(x), \quad 1\leq i\leq r-1,\\
\kappa'_r \rb{t*x} &= c^{-1}{a'_r}^2 \kappa'_r(x).
\end{split}
\end{equation}
For $\ell\in\N$, we define
\ba
A_w(\ell) &:= (\Z/p^\ell \Z)^\Delta \times (\Z/p^\ell \Z)^{\Delta_w}.
\ea
and
\ba
V_w(\ell) &:= \cbm{(\lambda, \lambda') \in A_w(\ell)^\times}{\begin{array}{l} \exists t \in \mc T \text{ such that } \kappa_i(t*x) = \lambda_i \kappa_i(x), \kappa'_j(t*x) = \lambda'_j \kappa'_j(x)\\ \text{for } x\in X(n), \; 1\leq i, j\leq r, \; \alpha_j\in \Delta_w \end{array}}.
\ea
\begin{lem}
We have $\vb{V_w(\ell)} = \rb{p^\ell \rb{1-p^{-1}}}^r$.
\end{lem}
\begin{proof}
For every $\lambda \in \rb{(\Z/p^\ell\Z)^\times}^\Delta$, we can find $t\in \mc T$ such that $\kappa_i(t*x) = \lambda_i\kappa_i(x)$ for $1\leq i\leq r$. Using \eqref{eq:kappap_i}, we find a unique $\lambda' \in \rb{(\Z/p^\ell\Z)^\times}^{\Delta_w}$ such that $\kappa'_i(t*x) = \lambda'_i\kappa_i(x)$ for $1\leq j\leq r$ with $\alpha_j\in\Delta_w$, and it is straightforward to check that $\lambda'$ is independent of the choice of $t\in\mc T$. Therefore 
\ba
\vb{V_w(\ell)} = \vb{\rb{(\Z/p^\ell\Z)^\times}^\Delta} = \rb{p^\ell \rb{1-p^{-1}}}^r
\ea 
as claimed.
\end{proof}
For a character $\theta: A_w(\ell) \to \C^\times$, we define
\ba
S_w\rb{\theta; \ell} = \sum\limits_{v \in V_w(\ell)} \theta(v). 
\ea

\begin{thm}\label{thm:Stevens4.10}
Let $n\in N\rb{\Q_p}$, and suppose $\ell$ is large enough such that the matrix entries of $u(x), u'(x)$ lie in $p^{-\ell}\Z_p / \Z_p$ for every $x\in X(n)$. Let $\psi = \psi_{n_1, \cdots, n_r}: U\rb{\Q_p}/U\rb{\Z_p} \to \C^\times$ and $\psi' = \psi_{n'_1, \cdots, n'_r}|_{U_n\rb{\Q_p}}: U_n\rb{\Q_p}/ U_n\rb{\Z_p} \to \C^\times$ be characters. Define the character $\theta_x: A_w(\ell) \to \C^\times$ by
\ba
\theta_x (\lambda, \lambda') = \prod\limits_{i=1}^r\e\rb{\lambda_i n_i \kappa_i(x)} \prod\limits_{\substack{i=1\\ w(\alpha_i)<0}}^r\e\rb{\lambda'_i n'_i \kappa'_i(x)}.
\ea
Then
\ba
\Kl_p\rb{n, \psi, \psi'} = \rb{p^\ell\rb{1-p^{-1}}}^{-r} \sum\limits_{x\in \mc T\bs X(n)} \mf N(x) S_w\rb{\theta_x; \ell},
\ea
where $\mf N(x) = \vb{\mc T*x}$ is the size of $\mc T$-orbit of $x\in X(n)$. 
\end{thm}
\begin{proof}
We rewrite the Kloosterman sum
\ba
\Kl_p\rb{n, \psi, \psi'} = &\sum\limits_{x\in \mc T\bs X(n)} \sum\limits_{y \in \mc T * x} \psi\rb{u(y)} \psi' \rb{u'(y)}\\
= & \sum\limits_{x\in \mc T\bs X(n)} \sum\limits_{y \in \mc T * x} \prod\limits_{i=1}^r\e\rb{n_i \kappa_i(y)} \prod\limits_{\substack{i=1\\ w(\alpha_i)<0}}^r\e\rb{n'_i \kappa'_i(y)}.
\ea
For $(\lambda,\lambda')\in V_w(\ell)$, we can find $t\in \mc T$ such that $\kappa_i(t*x) = \lambda_i\kappa_i(x)$, $\kappa'_j(t*x) = \lambda'_i\kappa'_j(x)$ for $x\in X(n)$, $1\leq i,j\leq r$, $w(\alpha_j)<0$. Hence
\ba
\sum\limits_{y \in \mc T * x} \prod\limits_{i=1}^r\e\rb{\lambda_i n_i \kappa_i(y)} \prod\limits_{\substack{i=1\\ w(\alpha_i)<0}}^r\e\rb{\lambda'_i n'_i \kappa'_i(y)}
&= \sum\limits_{y \in \mc T * x} \prod\limits_{i=1}^r\e\rb{n_i \kappa_i(t*y)} \prod\limits_{\substack{i=1\\ w(\alpha_i)<0}}^r\e\rb{n'_i \kappa'_i(t*y)}\\
&= \sum\limits_{y \in \mc T * x} \prod\limits_{i=1}^r\e\rb{n_i \kappa_i(y)} \prod\limits_{\substack{i=1\\ w(\alpha_i)<0}}^r\e\rb{n'_i \kappa'_i(y)}.
\ea
Summing over $V_w(\ell)$, we have
\ba
\vb{V_w(\ell)} \Kl_p(n,\psi,\psi') &= \sum\limits_{x\in \mc T\bs X(n)} \sum\limits_{y \in \mc T * x} \sum\limits_{(\lambda,\lambda') \in V_w(\ell)} \prod\limits_{i=1}^r\e\rb{\lambda_i n_i \kappa_i(y)} \prod\limits_{\substack{i=1\\ w(\alpha_i)<0}}^r\e\rb{\lambda'_i n'_i \kappa'_i(y)}\\
&= \sum\limits_{x\in \mc T\bs X(n)} \mf N(x) \sum\limits_{(\lambda,\lambda') \in V_w(\ell)} \prod\limits_{i=1}^r\e\rb{\lambda_i n_i \kappa_i(y)}\prod\limits_{\substack{i=1\\ w(\alpha_i)<0}}^r\e\rb{\lambda'_i n'_i \kappa'_i(y)}\\
&= \sum\limits_{x\in \mc T\bs X(n)} \mf N(x) S_w\rb{\theta_x; \ell}.
\ea
Dividing both sides by $\vb{V_w(\ell)}$ yields the statement.
\end{proof}

\section{$\Sp(4)$ Kloosterman sums} \label{section:Sp4Kloosterman}

Now we give explicit formulations for Kloosterman sums for $G = \Sp\rb{4, \Q_p}$, classified by the image $w(n)$ of the projection onto $W$. Fix $\psi = \psi_{m_1, m_2}$, $\psi' = \psi_{n_1, n_2}$, with $\psi_{m_1, m_2}$, $\psi_{n_1, n_2}$ as in \eqref{eq:psi4_def}.

\begin{prp}\label{prp:Stevens3.2}
\cite[Theorem 3.2]{Stevens1987} Let $n\in N\rb{\Q_p}$, and $\psi: U\rb{\Q_p}/U\rb{\Z_p} \to \C^\times$, $\psi': U_n\rb{\Q_p}/ U_n\rb{\Z_p}\to \C^\times$ be characters. If $t\in T\rb{\Z_p^\times}$, then
\ba
\Kl_p\rb{tn, \psi, \psi'} &= \Kl_p\rb{n, \psi_t, \psi'},\\
\Kl_p\rb{nt^{-1}, \psi, \psi'} &= \Kl_p\rb{n, \psi, \psi'_t},
\ea
where $\psi_t (u) = \psi(tut^{-1})$. 
\end{prp}

By \Cref{prp:Stevens3.2}, it suffices to consider Kloosterman sums $\Kl_p\rb{n, \psi, \psi'}$ for $n$ such that entries of $n$ are powers of $p$, and $X(n)$ is nonempty. To express the Kloosterman sums, we express the coset representatives for $X(n)$ in terms of Plücker coordinates, which were introduced in \cite{BFH1990,Man2020}. For $g = (g_{ij}) \in G = \Sp(4,\Q_p)$, we define Plücker coordinates
\ba
v_i &:= g_{3,i}, & &1\leq i\leq 4,\\
v_{ij} &:= g_{3,i}g_{4,j}-g_{3,j}g_{4,i}, & &1\leq i<j\leq 4.
\ea
The Plücker coordinates satisfy the following relations:
\begin{equation}\label{eq:Pl_rel}
\begin{split}
v_iv_{jk}-v_jv_{ik}+v_kv_{ij} &= 0, \quad 1\leq i<j<k\leq 4,\\
v_{13}+v_{24} &= 0.
\end{split}
\end{equation}
Hence, we can associate to every $g\in G(\Q_p)$ its Plücker coordinates
\ba
v = v_g = (v_1, v_2, v_3, v_4; v_{12}, v_{13}, v_{14}, v_{23}, v_{24}, v_{34}) \in \Q_p^{10}.
\ea
It follows from the definition that if $g\in G(\Z_p)$, then the corresponding Plücker coordinates $v_g$ are integral, and satisfy the coprimality conditions
\begin{align}\label{eq:Pl_cp}
(v_1,v_2,v_3,v_4) &= 1, & (v_{12},v_{13},v_{14},v_{23},v_{24},v_{34}) &= 1.
\end{align}

It is proved in \cite{Man2020} that there is a bijection
\bdd
\cb{\text{cosets } U(\Q_p)\bs G(\Z_p)} \rar[leftrightarrow] & \cbm{v_g \in \Z_p^{10}}{v_g \text{ satisfies \eqref{eq:Pl_rel} and \eqref{eq:Pl_cp}}}. 
\edd
In particular, this means the coset representatives for $X(n)$ can be described using Plücker coordinates. For notational convenience, for $n\in N(\Q_p)$ we write $X^v(n)$ for a complete system of coset representatives of $X(n)$, in terms of Plücker coordinates.

Now we give explicit formulations for Kloosterman sums $\Kl_p(n,\psi,\psi')$. Note that by \Cref{prp:Stevens3.2}, it suffices to consider the case $n = n_{w,r,s}$, where $n_{w,r,s}$ is given as in \eqref{eq:nwrs}. By looking at the Plücker coordinates, one deduces that $X(n_{w,r,s})$ is nonempty only if $r,s\geq 0$. Explicit formulations for $X^v(n_{w,r,s})$ are obtained by unfolding the conditions \eqref{eq:Pl_rel} and \eqref{eq:Pl_cp}, and are given in \cite{Man2020}, and we shall use the results from there directly.

\ber
\item $n = n_{\id,r,s}$. Then $X(n_{\id,r,s})$ is empty unless $r=s=0$, where $X(n_{\id,0,0}) = \cb{I_4}$ is a singleton. So the Kloosterman sum is trivial:
\ba
\Kl_p\rb{n_{\id,0,0}, \psi, \psi'} = 1.
\ea

\item $n = n_{s_\alpha,r,s}$. Then $X(n_{s_\alpha,r,s})$ is nonempty when $s=0$. In this case we have
\ba
X^v(n_{s_\alpha,r,0}) = \cb{(0,0,v_3,p^r;0,0,0,0,0,1)}, 
\ea
where $v_3 \pmod{p^r}$ such that $(v_3,p^r) = 1$. The corresponding Kloosterman sum is actually a $\GL(2)$ Kloosterman sum:
\ba
\Kl_p\rb{n_{s_\alpha,r,0}, \psi, \psi'} = S\rb{m_1, n_1; p^r}. 
\ea

\item $n = n_{s_\beta,r,s}$. Then $X(n_{s_\beta,r,s})$ is nonempty when $r=0$. In this case we have
\ba
X^v(n_{s_\beta,0,s}) = \cb{(0,0,1,0;0,0,0,p^s,0,v_{34})},
\ea
where $v_{34}\pmod{p^s}$ such that $(v_{34},p^s) = 1$. The corresponding Kloosterman sum is actually a $\GL(2)$ Kloosterman sum:
\ba
\Kl_p\rb{n_{s_\beta,0,s}, \psi, \psi'} = S\rb{m_2, n_2; p^s}. 
\ea

\item $n = n_{s_\alpha s_\beta, r,s}$. Then $X(n_{s_\alpha s_\beta,r,s})$ is nonempty when $r\geq s$. Unfolding the conditions \eqref{eq:Pl_rel} and \eqref{eq:Pl_cp}, we compute
\ba
X^v(n_{s_\alpha s_\beta,r,s}) = \cb{(0,p^r,v_3,v_4;0,0,0,p^s,0,-v_4p^{s-r})},
\ea
where $v_3,v_4\pmod{p^r}$ such that $(v_4,p^r) = p^{r-s}$ and $(v_3,p^{r-s}) = 1$. We write $v_4 = v'_4 p^{r-s}$, so $(v'_4,p^s) = 1$. Bruhat decomposition gives
\ba
x = &\bp 1 & \beta_1 & \beta_2 & \beta_3\\ & 1 & \beta_4 & \beta_5\\ &&1\\&&-\beta_1 & 1\ep \bp &&&-p^{-r}\\ p^{r-s}\\ &p^r\\ &&p^{s-r}\ep \bp 1 &&& v_3p^{-r}\\ &1&v_3p^{-r}&v'_4 p^{-s}\\&&1\\&&&1\ep\\
= &\bp \beta_1 p^{r-s} & \beta_2p^r & \beta_2v_3 + \beta_3 p^{s-r} & \beta_2 v'_4 p^{r-s} + \beta_1 v_3 p^{-s} - p^{-r}\\ 
p^{r-s} & \beta_4 p^r & \beta_4 v_3 + \beta_5 p^{s-r} & \beta_4 v'_4 p^{r-s} + v_3 p^{-s}\\
 0 & p^r & v_3 & v'_4 p^{r-s}\\ 
0 & -\beta_1 p^r & -\beta_1 v_3 + p^{s-r}  & -\beta_1 v'_4 p^{r-s}\ep.
\ea
The entry $-\beta_1 v_3 + p^{s-r}$ being an integer says $\beta_1 \equiv \ol{v_3} p^{s-r} \pmod{1}$. The entry $\beta_4 v'_4 p^{r-s} + v_3 p^{-s}$ being an integer says $\beta_4 \equiv -\ol{v'_4} v_3 p^{-r} \pmod{p^{s-r}}$. Write $\beta_4 = -\ol{v'_4} v_3 p^{-r} + \gamma_4 p^{s-r}$ for some $\gamma_4\in\Z$. The entry $\beta_4 v_3 + \beta_5 p^{s-r}$ being an integer says
$\gamma_4 v_3 + \beta_5 \equiv \ol{v'_4} v_3^2 p^{-s} \pmod{p^{r-s}}$, hence $\beta_5 \equiv \ol{v'_4} v_3^2 p^{-s} \pmod{1}$. After writing $v_4$ for $v'_4$, the Kloosterman sum is given by 
\ba
\Kl_p \rb{n_{s_\alpha s_\beta, r, s}, \psi, \psi'} = \sum\limits_{\substack{v_4 \ppmod{p^s}\\ (v_4, p^s) = 1}} \sum\limits_{\substack{v_3 \ppmod{p^r}\\ (v_3, p^{r-s}) = 1}}\e\rb{\frac{m_1\ol{v_3}p^s}{p^r}}\e\rb{\frac{m_2 \ol{v_4} v_3^2 + n_2 v_4}{p^s}}.
\ea

\item $n = n_{s_\beta s_\alpha,r,s}$. Then $X(n_{s_\beta s_\alpha,r,s})$ is nonempty when $s\geq 2r$. Unfolding the conditions \eqref{eq:Pl_rel} and \eqref{eq:Pl_cp}, we compute
\ba
X^v(n_{s_\beta s_\alpha,r,s}) = \cb{(0,0,-v_{24}p^{r-s},p^r;0,-v_{24},p^s,-v_{24}p^{-s},v_{24},v_{34})},
\ea
where $v_{24}, v_{34}\pmod{p^s}$ such that $(v_{24},p^s) = p^{s-r}$ and $(v_{34}, p^{s-2r}) = 1$. We write $v_{24} = v'_{24} p^{s-r}$, so $(v'_{24},p^r) = 1$. Bruhat decomposition gives
\ba
x = &\bp 1 & \beta_1 & \beta_2 & \beta_3\\ & 1 & \beta_4 & \beta_5\\ &&1\\&&-\beta_1 & 1\ep \bp & p^{-r}\\ && p^{r-s}\\ &&& p^r\\ -p^{s-r}\ep \bp 1 & v'_{24} p^{-r} & v_{34} p^{-s}\\ &1\\&&1\\&&-v'_{24} p^{-r} & 1\ep\\
= &\bp -\beta_3 p^{s-r} & -\beta_3 v'_{24} p^{s-2r} + p^{-r} & -\beta_2 v'_{24} - \beta_3 v_{34} p^{-r} + \beta_1 p^{r-s} & \beta_2 p^r\\
-\beta_5 p^{s-r} & -\beta_5 v'_{24} p^{s-2r} & -\beta_4 v'_{24} - \beta_5 v_{34} p^{-r} + p^{r-s} & \beta_4 p^r\\
0 & 0 & -v'_{24} & p^r\\
-p^{s-r} & -v'_{24} p^{s-2r} & \beta_1 v'_{24} - v_{34} p^{-r} & -\beta_1 p^r\ep.
\ea
The entry $\beta_1 v'_{24} - v_{34} p^{-r}$ being an integer says $\beta_1 \equiv \ol{v'_{24}} v_{34} p^{-r} \pmod{1}$. The entry $\beta_4 p^r$ being an integer says $\beta_4 = \beta'_4 p^{-r}$ for some $\beta'_4\in\Z$. The entry $-\beta_4 v'_{24} - \beta_5 v_{34} p^{-r} + p^{r-s}$ being an integer says $\beta'_4 v'_{24} + \beta_5 v_{34} \equiv p^{2r-s} \pmod{p^r}$, hence $\beta_5 \equiv \ol{v_{34}} p^{2r-s} \pmod{1}$. After writing $v_{24}$ for $v'_{24}$, the Kloosterman sum is given by
\ba
\Kl_p \rb{n_{s_\beta s_\alpha, r, s}, \psi, \psi'} = \sum\limits_{\substack{v_{24} \ppmod{p^r}\\ (v_{24}, p^r) = 1}} \sum\limits_{\substack{v_{34}\ppmod{p^s}\\ (v_{34}, p^{s-2r}) = 1}}\e\rb{\frac{m_1 \ol{v_{24}} v_{34} + n_1 v_{24}}{p^r}}\e\rb{\frac{m_2 \ol{v_{34}} p^{2r}}{p^s}}.
\ea
\begin{rmk}
This Kloosterman sum is related to a $\GL(3)$ Kloosterman sum. Precisely, following the notation in \cite[(4.3)]{BFG1988}, we have
\ba
\Kl_p \rb{n_{s_\beta s_\alpha, r, s}, \psi, \psi'} = p^r S\rb{n_1, m_1, m_2; p^r, p^{s-r}}.
\ea
A non-trivial bound for $\Kl_p \rb{n_{s_\beta s_\alpha, r, s}, \psi, \psi'}$ then follows from Larsen \cite[Appendix]{BFG1988}.
\end{rmk}

\item $n=n_{s_\alpha s_\beta s_\alpha,r,s}$. Then $X(n_{s_\alpha s_\beta s_\alpha,r,s})$ is nonempty when $2r\geq s$. Unfolding the conditions \eqref{eq:Pl_rel} and \eqref{eq:Pl_cp}, we compute
\ba
X(n_{s_\alpha s_\beta s_\alpha, r,s}) = \cb{\rb{p^r,v_2,v_3,v_4; 0, -v_2 p^{s-r}, p^s, -v_2^2 p^{s-2r}, v_2 p^{s-r}, (p^rv_3+v_2v_4)p^{s-2r}}},
\ea
where $v_2, v_3, v_4 \pmod{p^r}$, such that $(v_2,v_3,v_4,p^r) = 1$, and if $d := (v_2, p^r)$, then $(d^2, p^rv_3+v_2v_4) = p^{2r-s}$. Let $d = p^{r-a}$. Then $a$ satisfies $s-r\leq a \leq s/2$. We write $v_2 = v'_2 p^{r-a}$, so $(v'_2, p^a) = 1$. Bruhat decomposition gives
\ba
x = &\bp 1 & \beta_1 & \beta_2 & \beta_3\\ & 1 & \beta_4 & \beta_5\\ &&1\\&&-\beta_1 & 1\ep \bp &&-p^{-r}\\&p^{r-s}\\p^r\\&&&p^{s-r}\ep \bp 1 & v'_2 p^{-a} & v_3 p^{-r} & v_4 p^{-r}\\ & 1 & v_4 p^{-r} \\ &&1\\&&-v'_2 p^{-a} & 1\ep\\
= & \bp \beta_2 p^r & \beta_2 v'_2 p^{r-a} + \beta_1 p^{r-s} & \beta_2 v_3 - \beta_3 v'_2 p^{s-a-r} + \beta_1 v_4 p^{-s} - p^{-r} & \beta_2 v_4 + \beta_3 p^{s-r}\\
\beta_4 p^r & \beta_4 v'_2 p^{r-a} + p^{r-s} & \beta_4 v_3 - \beta_5 v'_2 p^{s-a-r} + v_4 p^{-s} & \beta_4 v_4 + \beta_5 p^{s-r}\\
p^r & v'_2 p^{r-a} & v_3 & v_4\\
-\beta_1 p^r & -\beta_1 v'_2 p^{r-a} & -\beta_1 v_3 - v'_2 p^{s-a-r} & -\beta_1 v_4 + p^{s-r}\ep.
\ea

The entry $-\beta_1 v'_2 p^{r-a}$ being an integer says $\beta_1 = \beta'_1 p^{a-r}$ for some $\beta'_1 \in \Z$. Entries $-\beta_1 v_3 - v'_2 p^{s-a-r}$ and $-\beta_1 v_4 + p^{s-r}$ being integers says
\begin{align}
\beta'_1 v_3 &\equiv -v'_2 p^{s-2a} \pmod{p^{r-a}}, & \beta'_1 v_4 &\equiv p^{s-a} \pmod{p^{r-a}}.
\end{align}
As $\rb{v_3, v_4, p^{r-a}} = 1$, these equations determine $\beta_1$ uniquely modulo 1.

The entry $\beta_4 v'_2 p^{r-a} + p^{r-s}$ being an integer says $\beta_4 \equiv -\ol{v'_2} p^{a-s} \pmod{p^{a-r}}$. Write $\beta_4 = -\ol{v'_2} p^{a-s} + \gamma_4 p^{a-r}$ for some $\gamma_4 \in \Z$. Then $\beta_4 v_3 - \beta_5 v'_2 p^{s-a-r} + v_4 p^{-s}$ being an integer says
\begin{align}
-\ol{v'_2} v_3 p^a + \gamma_4 v_3 p^{s+a-r} - \beta_5 v'_2 p^{2s-a-r} + v_4 \equiv 0 \pmod{p^s}.
\end{align}
Write $\beta_5 = \beta'_5 p^{a+r-2s}$ for some $\beta'_5 \in \Z$. Then we solve
\begin{align}\label{eq:Kl_aba_betap5_sol}
\beta'_5 \equiv -\ol{v'_2}^2 v_3 p^a + \gamma_4 \ol{v'_2} v_3 p^{s+a-r} + \ol{v'_2} v_4 \pmod{p^s}. 
\end{align}
Then $\beta_4 v_4 + \beta_5 p^{s-r}$ being an integer says
\begin{align} \label{eq:Kl_aba_gamma4}
\gamma_4 \rb{p^a v_3 + v'_2 v_4} p^{s+a-r} \equiv v_3 p^{2a} \pmod{p^s}.
\end{align}
Recall that $\rb{p^{r-a}, p^a v_3 + v'_2 v_4} = p^{r+a-s}$. Hence, unless $a = \frac{s}{2}$, we can write $p^a v_3 + v'_2 v_4 = V' p^{r+a-s}$, with $(V',p) = 1$. Then we solve \eqref{eq:Kl_aba_gamma4}:
\ba
\gamma_4 \equiv \ol{V'} v_3 \pmod{p^{s-2a}}. 
\ea
Putting back to \eqref{eq:Kl_aba_betap5_sol} gives
\ba
\beta'_5 \equiv -\ol{v'_2}^2 v_3 p^a + \ol{V' v'_2} v_3^2 p^{s+a-r} + \ol{v'_2} v_4 \pmod{p^{2s-a-r}},
\ea
hence $\beta_5$ is uniquely determined modulo 1. When $a=\frac{s}{2}$, $\gamma_4$ can be arbitrary, and we have
\ba
\beta'_5 \equiv -\ol{v'_2}^2 v_3 p^a + \ol{v'_2} v_4 \pmod{p^{2s-a-r}},
\ea
hence $\beta_5$ is also uniquely determined modulo 1 in this case. 

So, after writing $u$ for $\beta_5 p^s$, the Kloosterman sum is given by
\ba
\Kl_p\rb{n_{s_\alpha s_\beta s_\alpha, r, s}, \psi, \psi'} = &\sum\limits_{s-r \leq a \leq s/2}\sum\limits_{\substack{v_2, v_3, v_4 \ppmod{p^r}\\ v_2 = v'_2 p^{r-a}, \; (v'_2, p^a) = 1\\ (v_3, v_4, p^{r-a}) = 1\\ \rb{p^{r-a}, p^a v_3 + v'_2 v_4} = p^{r+a-s}}}\e\rb{\frac{m_1\hat v_2 + n_1v_2}{p^r}}\e\rb{\frac{m_2 u}{p^s}},
\ea
where $\hat v_2$ is chosen modulo $p^r$ such that
\begin{align}\label{eq:abaKloosterman_v2hat}
\hat v_2 v_3 &\equiv -v'_2 p^{s-a} \pmod{p^r}, & \hat v_2 v_4 \equiv p^s \pmod{p^r},
\end{align}
and 
\begin{align}\label{eq:abaKloosterman_u}
u \equiv \begin{cases} -\ol{v'_2}^2 v_3 p^{2a+r-s} + \ol{V' v'_2} v_3^2 p^{2a} + \ol{v'_2} v_4 p^{a+r-s} \pmod{p^s} & \text{if } a<\frac{s}{2},\\
-\ol{v'_2}^2 v_3 p^{2a+r-s} + \ol{v'_2} v_4p^{a+r-s} \pmod{p^s} & \text{if } a=\frac{s}{2},\end{cases}
\end{align}
where $V' = p^{s-r-a} \rb{p^a v_3 + v'_2 v_4}$.

\item $n = n_{s_\beta s_\alpha s_\beta,r,s}$. Then $X(n_{s_\beta s_\alpha s_\beta,r,s})$ is nonempty when $s\geq r$. Unfolding the conditions \eqref{eq:Pl_rel} and \eqref{eq:Pl_cp}, we compute
\ba
X(n_{s_\beta s_\alpha s_\beta, r,s}) = \cb{\rb{0,p^r, v_{13} p^{r-s}, v_{14} p^{r-s}; p^s, v_{13}, v_{14}, v_{23}, -v_{13}, -(v_{13}^2+v_{14}v_{23}) p^{-s}}},
\ea
where $v_{13}, v_{14}, v_{23} \pmod{p^s}$, such that $(v_{13}, v_{14}, p^s) = p^{s-r}$, $(v_{14}, p^s) \mid v_{13}^2$, and $(p^{s-r}, v_{23}, v_{34}) = 1$. Recall that $v_{34} = -(v_{13}^2+v_{14}v_{23})p^{-s}$. We write $v_{13} = v'_{13} p^{s-r}$, $v_{14} = v'_{14} p^{s-r}$, so $(v'_{13}, v'_{14}, p^r) = 1$. Bruhat decomposition gives
\ba
x = &\bp 1 & \beta_1 & \beta_2 & \beta_3\\ & 1 & \beta_4 & \beta_5\\ &&1\\&&-\beta_1 & 1\ep \bp &&& -p^{-r}\\ &&p^{r-s}\\ &p^r\\ -p^{s-r}\ep \bp 1 && -v_{23} p^{-s} & v'_{13} p^{-r}\\ &1& v'_{13} p^{-r} & v'_{14} p^{-r}\\ &&1\\ &&&1\ep\\
= &\bp -\beta_3 p^{s-r} & \beta_2 p^r & \beta_2 v'_{13} + \beta_1 p^{r-s} + \beta_3 v_{23} p^{-r} & \beta_2 v'_{14} - \beta_3 v'_{13} p^{s-2r} - p^{-r}\\
-\beta_5 p^{s-r} & \beta_4 p^r & \beta_4 v'_{13} + \beta_5 v_{23} p^{-r} + p^{r-s} & \beta_4 v'_{14} - \beta_5 v'_{13} p^{s-2r}\\
0 & p^r & v'_{13} & v'_{14}\\
-p^{s-r} & -\beta_1 p^r & -\beta_1 v'_{13} + v_{23} p^{-r} & -\beta_1 v'_{14} - v'_{13} p^{s-2r}\ep.
\ea

The entry $-\beta_1 p^r$ being an integer says $\beta_1 = \beta'_1 p^{-r}$ for $\beta'_1 \in \Z$. Entries $-\beta_1 v'_{13} + v_{23} p^{-r}$ and $-\beta_1 v'_{14} - v'_{13} p^{s-2r}$ being integers says
\begin{align}
\beta'_1 v'_{13} &\equiv v_{23} \pmod{p^r}, & \beta'_1 v'_{14} &\equiv - v'_{13} p^{s-r} \pmod{p^r}.
\end{align}
As $(v'_{13}, v'_{14},p^r)=1$, this determines $\beta_1$ uniquely modulo 1.

Entries $\beta_4 p^r$ and $-\beta_5 p^{s-r}$ being integers says $\beta_4 = \beta'_4 p^{-r}$ and $\beta_5 = \beta'_5 p^{r-s}$ for some $\beta'_4, \beta'_5 \in \Z$. The entry $\beta_4 v'_{13} + \beta_5 v_{23} p^{-r} + p^{r-s}$ being an integer says
\begin{align}\label{eq:Kl_bab_1}
\beta'_4 v'_{13} p^{s-r} + \beta'_5 v_{23} + p^r \equiv 0 \pmod{p^s},
\end{align}
which implies
\begin{align}\label{eq:Kl_bab_sol_1}
\beta'_5 v_{23} \equiv -p^r \pmod{p^{s-r}}.
\end{align}
The entry $\beta_4 v'_{14} - \beta_5 v'_{13} p^{s-2r}$ being an integer says
\begin{align}\label{eq:Kl_bab_2}
\beta'_4 v'_{14} p^{s-r} - \beta'_5  v'_{13} p^{s-r} \equiv 0 \pmod{p^s}. 
\end{align}
Then, $v'_{13}$ times \eqref{eq:Kl_bab_2} minus $v'_{14}$ times \eqref{eq:Kl_bab_1} gives
\begin{align}
\beta'_5 (- {v'_{13}}^2 p^{s-r} - v'_{14} v_{23}) &\equiv p^r v'_{14} \pmod{p^s}\nonumber\\
\beta'_5 p^r v_{34} &\equiv p^r v'_{14} \pmod{p^s}\nonumber\\
\beta'_5 v_{34} & \equiv v'_{14} \pmod{p^{s-r}}. \label{eq:Kl_bab_sol_2}
\end{align}
As $\rb{p^{s-r}, v_{23}, v_{34}} = 1$, \eqref{eq:Kl_bab_sol_1} and \eqref{eq:Kl_bab_sol_2} determine $\beta_5$ uniquely modulo 1.

So, after writing $u$ for $\beta_1 p^r$ and $\hat v_{14}$ for $\beta_5 p^s$, the Kloosterman sum is given by
\ba
\Kl_p\rb{n_{s_\beta s_\alpha s_\beta, r, s}, \psi, \psi'} =  \sum\limits_{\substack{v_{13}, v_{14}, v_{23} \ppmod{p^s}\\ \rb{p^s, v_{13}, v_{14}} = p^{s-r}\\ \rb{p^s, v_{14}} \mid v_{13}^2\\ \rb{p^{s-r}, v_{23}, v_{34}} = 1}}\e\rb{\frac{m_1 u}{p^r}}\e\rb{\frac{m_2 \hat v_{14} + n_2 v_{14}}{p^s}},
\ea
where $u$ is chosen modulo $p^r$ such that 
\begin{align}\label{eq:babKloosterman_u}
u v_{13} p^{r-s} &\equiv v_{23} \pmod{p^r}, & u v_{14} p^{r-s} &\equiv - v_{13} \pmod{p^r},
\end{align}
and $\hat v_{14}$ is chosen modulo $p^s$ such that
\begin{align}\label{eq:babKloosterman_v14hat}
\hat v_{14} v_{23} &\equiv -p^{2r} \pmod{p^s}, & \hat v_{14} v_{34} & \equiv v_{14} p^{2r-s} \pmod{p^s}.
\end{align}

\item $n=n_{w_0,r,s}$. Then $X(n_{w_0,r,s})$ is nonempty whenever $r,s\geq 0$. Unfolding the conditions \eqref{eq:Pl_rel} and \eqref{eq:Pl_cp}, we compute
\ba
X(n_{w_0,r,s}) = \cb{\rb{p^r,v_2,v_3,v_4;p^s,v_{13},v_{14},(v_2v_{13}-v_3p^s)p^{-r},-v_{13},(v_3v_{14}-v_4v_{13})p^{-r}}}, 
\ea
where $v_2,v_3,v_4\pmod{p^r}$, $v_{13},v_{14}\pmod{p^s}$, such that $v_{13}p^r+v_2v_{14}-v_4p^s = 0$, $(v_2,v_3,v_4,p^r)=1$, and $(v_{13},v_{14},v_{23},v_{34},p^s)=1$. Recall that
\ba
v_{23} = (v_2v_{13}-v_3p^s)p^{-r}, & \quad v_{34} = (v_3v_{14}-v_4v_{13})p^{-r}.
\ea
Bruhat decomposition gives
\ba
x = &\bp 1 & \beta_1 & \beta_2 & \beta_3\\ & 1 & \beta_4 & \beta_5\\ &&1\\&&-\beta_1 & 1\ep \bp && -p^{-r}\\ &&&-p^{r-s}\\ p^r\\ &p^{s-r}\ep \bp 1 & v_2 p^{-r} & v_3 p^{-r} & v_4 p^{-r}\\ & 1 & v_{13} p^{-s} & v_{14} p^{-s}\\ &&1\\ &&-v_2 p^{-r} & 1\ep\\
= &\scalebox{0.9}{$\displaystyle
\bp \beta_2 p^r & \beta_2 v_2 + \beta_3 p^{s-r} & \beta_2 v_3 + \beta_3 v_{13} p^{-r} + \beta_1 v_2 p^{-s} - p^{-r} & \beta_2 v_4 - \beta_1 p^{r-s} + \beta_3 v_{14} p^{-r}\\
\beta_4 p^r & \beta_4 v_2 + \beta_5 p^{s-r} & \beta_4 v_3 + \beta_5 v_{13} p^{-r} + v_2 p^{-s} & \beta_4 v_4 + \beta_5 v_{14} p^{-r} - p^{r-s}\\
p^r & v_2 & v_3 & v_4\\
-\beta_1 p^r & -\beta_1 v_2 + p^{s-r} & -\beta_1 v_3 + v_{13} p^{-r} & -\beta_1 v_4 + v_{14} p^{-r}\ep.$}
\ea

The entry $-\beta_1 p^r$ being an integer says $\beta_1 = \beta'_1 p^{-r}$ for some $\beta'_1\in\Z$. The last row of $\gamma$ being integral gives
\begin{align}
\beta'_1 v_2 &\equiv p^s \pmod{p^r}, & \beta'_1 v_3 &\equiv v_{13} \pmod{p^r}, & \beta'_1 v_4 &\equiv v_{14} \pmod{p^r}. 
\end{align}
As $\rb{p^r, v_2, v_3, v_4} = 1$, these equations determine $\beta_1$ uniquely modulo 1.

The entry $\beta_4 p^r$ being an integer says $\beta_4 = \beta'_4 p^{-r}$ for some $\beta'_4 \in \Z$. Then $\beta_4 v_2 + \beta_5 p^{s-r}$ being an integer says 
\begin{align}\label{eq:Kl_abab_v12}
\beta'_4 v_2 + \beta_5 p^s \equiv 0 \pmod{p^r}. 
\end{align}
In particular, this means $\beta_5 = \beta'_5 p^{-s}$ for some $\beta'_5 \in \Z$. The entries $\beta_4 v_3 + \beta_5 v_{13} p^{-r} + v_2 p^{-s}$ and $\beta_4 v_4 + \beta_5 v_{14} p^{-r} - p^{r-s}$ being integers says
\begin{align}
\beta'_4 v_3 p^s + \beta'_5 v_{13} + v_2 p^r &\equiv 0 \pmod{p^{r+s}}, \label{eq:Kl_abab_v13}\\
\beta'_4 v_4 p^s + \beta'_5 v_{14} - p^{2r} &\equiv 0 \pmod{p^{r+s}} \label{eq:Kl_abab_v14}.
\end{align}
In particular we deduce 
\begin{align}
\beta'_5 v_{13} + v_2 p^r &\equiv 0 \pmod{p^s},\\
\beta'_5 v_{14} - p^{2r} &\equiv 0 \pmod{p^s}.
\end{align}
Then, $v_2$ times \eqref{eq:Kl_abab_v13} minus $v_3 p^s$ times \eqref{eq:Kl_abab_v12} gives
\begin{align*}
\beta'_5 \rb{v_2v_{13} - v_3p^s} + v_2^2 p^r \equiv 0 \pmod{p^{r+s}},
\end{align*}
which implies
\begin{align}
\beta'_5 v_{23} + v_2^2 \equiv 0 \pmod{p^s}.
\end{align}
Similarly, $v_3$ times \eqref{eq:Kl_abab_v13} minus $v_4$ times \eqref{eq:Kl_abab_v14} gives
\begin{align*}
&\beta'_5 \rb{v_3v_{14}-v_4v_{13}} - p^r\rb{v_3 p^r + v_2v_4} \equiv 0 \pmod{p^{r+s}},
\end{align*}
which implies
\begin{align}
\beta'_5 v_{34} \equiv \rb{v_3 p^r + v_2 v_4}\pmod{p^s}. 
\end{align}
In summary, $\beta'_5$ satisfies the following equations:
\ba
\beta'_5 v_{13} &\equiv - v_2 p^r  \pmod{p^s}, & \beta'_5 v_{14} &\equiv p^{2r} \pmod{p^s},\\
\beta'_5 v_{23} &\equiv -v_2^2 \pmod{p^s}, & \beta'_5 v_{34} &\equiv v_3 p^r + v_2v_4 \pmod{p^s}.
\ea
As $\rb{p^s, v_{13}, v_{14}, v_{23}, v_{34}} = 1$, these equations determine $\beta_5$ uniquely modulo 1.

So, after writing $\hat v_2$ for $\beta_1 p^r$ and $\hat v_{14}$ for $\beta_2 p^s$, the Kloosterman sum is given by
\ba
\Kl_p \rb{n_{w_0, r, s}, \psi, \psi'} = \sum\limits_{\substack{v_2, v_3, v_4\ppmod{p^r}\\ v_{13}, v_{14}\ppmod{p^s}\\ v_{13}p^r + v_2v_{14} - v_4p^s = 0\\ (p^r, v_2, v_3, v_4) = 1\\ (p^s, v_{13}, v_{14}, v_{23}, v_{34}) = 1}}\e\rb{\frac{m_1\hat v_2 + n_1 v_2}{p^r}}\e\rb{\frac{m_2\hat v_{14} + n_2 v_{14}}{p^s}},
\ea
where $\hat v_2$ is chosen modulo $p^r$ such that
\begin{align}\label{v2hatcong}
\hat v_2 v_2 \equiv p^s \pmod{p^r}, \quad \hat v_2 v_3 \equiv v_{13} \pmod{p^r}, \quad \hat v_2 v_4 \equiv v_{14} \pmod{p^r};
\end{align}
and $\hat v_{14}$ chosen modulo $p^s$ such that
\begin{align}
\begin{aligned}
\hat v_{14}v_{13} &\equiv -v_2p^r \pmod{p^s}, & \hat v_{14}v_{14} &\equiv p^{2r} \pmod{p^s},\\
\hat v_{14}v_{23} &\equiv -v_2^2 \pmod{p^s}, & \hat v_{14} v_{34} &\equiv v_3p^r+v_2v_4\pmod{p^s}.
\end{aligned}
\end{align}
\ee

Now we give a few reduction formulae for Kloosterman sums, which are straightforward to prove.
\begin{prp}
Let $\psi = \psi_{m_1, m_2}$, $\psi' = \psi_{n_1, n_2}$. Then 
\ba
\Kl_p\rb{n_{w_0, r, 0}, \psi, \psi'} &= S\rb{m_1, n_1; p^r}, & \Kl_p\rb{n_{w_0, 0, s}, \psi, \psi'} &= S\rb{m_2, n_2; p^s},\\
\Kl_p\rb{n_{s_\alpha s_\beta s_\alpha, r, 0}, \psi, \psi'} &= S\rb{m_1,0;p^r}, & \Kl_p\rb{n_{s_\beta s_\alpha s_\beta, 0, s}, \psi, \psi'} &= S\rb{0,m_2;p^s},\\
\Kl_p\rb{n_{s_\alpha s_\beta, r, 0}, \psi, \psi'} &= S\rb{m_1,0;p^r}, & \Kl_p\rb{n_{s_\beta s_\alpha, 0, s}, \psi, \psi'} &= S\rb{0,m_2;p^s}.
\ea
\end{prp}

We end the section by proving that the Kloosterman sum attached to the long Weyl element $w_0$ is symmetric with respect to characters $\psi, \psi'$. Note that this holds for $G = \Sp(2r)$ in general. 

\begin{prp}\label{prp:w0Kloosterman_swap}
Let $G = \Sp(2r,\Q_p)$, and $n\in N\rb{\Q_p}$, such that $w(n) = w_0$ is the long Weyl element. Let $\psi, \psi': U\rb{\Q_p} / U\rb{\Z_p} \to \C^\times$ be characters. Then
\ba
\Kl_p\rb{n, \psi, \psi'} = \Kl_p\rb{n, \psi', \psi}.
\ea
\end{prp}
\begin{proof}
The definition of Kloosterman sums reads
\ba
\Kl_p\rb{n, \psi, \psi'} = \sum\limits_{x\in X(n)} \psi\rb{u(x)} \psi'\rb{u'(x)}.
\ea
The key idea of the proof is to find a bijection $X(n) \to X(n)$, $x\mapsto \tilde x$ such that $\psi(u(\tilde x)) = \psi(u'(x))$ and $\psi'(u'(\tilde x)) = \psi'(u(x))$. Since $w(n) = w_0$ is the long Weyl element, $n\in N(\Q_p)$ is of the form
\ba
n &= \bp & -D^{-1}\\ D \ep, & D&=\diag(d_1,\cdots, d_r), & &d_i\in\Q_p^\times.
\ea
Let $x\in X(n)$, and suppose
\ba
u(x) &= \bp U & S\\ & (U^{-1})^T\ep \in U(\Q_p), & u'(x) &= \bp U' & S'\\ & ({U'}^{-1})^T\ep \in U(\Q_p). 
\ea
Then we have
\ba
x = \bp SDU' & SDS' - UD^{-1}({U'}^{-1})^T\\ (U^{-1})^T D U' & (U^{-1})^T D S'\ep \in G(\Z_p).
\ea
Now set
\ba
\tilde u &= \bp ({\tilde U'}{}^{-1})^T & \tilde S'\\ & \tilde U'\ep, & \tilde u' &= \bp ({\tilde U}^{-1})^T & \tilde S\\ & \tilde U\ep,
\ea
where
\ba
\tilde U_{ij} &:= (-1)^{i-j} U_{ji}, & \tilde S_{ij} &:= (-1)^{i-j} S_{ji}, & \tilde U'_{ij} &:= (-1)^{i-j} U'_{ji}, & \tilde S'_{ij} &:= (-1)^{i-j} S'_{ji}.
\ea
It is straightforward to verify that $\tilde u, \tilde u' \in U(\Q_p)$. Now set
\ba
\tilde x = \tilde u n \tilde u' = \bp \tilde S' D ({\tilde U}^{-1})^T & \tilde S' D \tilde S - ({\tilde U'}{}^{-1})^T D ^{-1} \tilde U\\ \tilde U' D ({\tilde U}^{-1})^T & \tilde U' D \tilde S\ep \in G(\Q_p).
\ea
Now observe
\ba
\rb{\tilde S' D ({\tilde U}^{-1})^T}_{ij} = \sum\limits_k \tilde S'_{ik} d_k ({\tilde U}^{-1})^T_{kj} = \sum\limits_k (-1)^{i+j} (U^{-1})^T_{jk} d_k S'_{ki} = (-1)^{i+j} \rb{(U^{-1})^T D S'}_{ji} \in \Z_p,
\ea
and similarly
\ba
\rb{\tilde S' D \tilde S - ({\tilde U'}{}^{-1})^T D ^{-1} \tilde U}_{ij} &= (-1)^{i+j} \rb{SDS' - UD^{-1}({U'}^{-1})^T}_{ji} \in \Z_p,\\
\rb{\tilde U' D ({\tilde U}^{-1})^T}_{ij} &= (-1)^{i+j} \rb{(U^{-1})^T D U'}_{ji} \in \Z_p. 
\ea
Hence $\tilde x \in G(\Z_p)$. Moreover, we may directly verify that $\alpha_i(\tilde u) = \alpha_i(u')$, $\alpha_i(\tilde u') = \alpha_i(u)$ for $1\leq i\leq r$. So $\psi(u(\tilde x)) = \psi(u'(x))$ and $\psi'(u'(\tilde x)) = \psi'(u(x))$. Finally, using the bijection $X(n) \to X(n)$, $x\mapsto \tilde x$, we deduce that
\begin{align*}
\Kl_p\rb{n, \psi, \psi'} &= \sum\limits_{x\in X(n)} \psi\rb{u(x)} \psi'\rb{u'(x)}\\
&= \sum\limits_{x\in X(n)} \psi\rb{u(\tilde x)} \psi'\rb{u'(\tilde x)}\\
&= \sum\limits_{x\in X(n)} \psi'\rb{u(x)} \psi\rb{u'(x)} = \Kl_p\rb{n, \psi', \psi}.\qedhere
\end{align*}
\end{proof}

\section{Bounds for $\Sp(4)$ Kloosterman sums} \label{section:Sp4Kloosterman_bound}

Fix $\psi = \psi_{m_1,m_2}$, $\psi' = \psi_{n_1,n_2}$ as in \eqref{eq:psi4_def}. We first establish non-trivial bounds for local Kloosterman sums $\Kl_p\rb{n_{w,r,s}, \psi, \psi'}$, that is, prove \Cref{thm:local_bound}.

We start with the local bounds. For $\Kl_p(n_{\id,0,0},\psi,\psi')$, there is nothing to prove. Meanwhile, $\Kl_p\rb{n_{s_\alpha,r,0}, \psi, \psi'}$ and $\Kl_p\rb{n_{s_\beta,0,s}, \psi, \psi'}$ are just $\GL(2)$ Kloosterman sums. A well-known bound for $\GL(2)$ Kloosterman sums is given by \cite{Smith1980}
\begin{align}\label{eq:Kloosterman_GL2bound}
\vb{S(\mu, \nu; p^k)} \leq 2 p^{k/2} (\vb{\mu}_p^{-1}, \vb{\nu}_p^{-1}, p^k)^{1/2}.
\end{align}
So
\ba
\vb{\Kl_p\rb{n_{s_\alpha, r}, \psi, \psi'}} = S(m_1,n_1,p^r) &\ll p^{r/2}(m_1,n_1,p^r)^{1/2}, \\
\vb{\Kl_p\rb{n_{s_\beta, s}, \psi, \psi'}} = S(m_2,n_2,p^s) &\ll p^{s/2} (m_2,n_2,p^s)^{1/2}
\ea
as claimed.

\subsection{Bound for $\Kl_p\rb{n_{s_\alpha s_\beta, r, s}, \psi, \psi'}$} 
We recall
\ba
\Kl_p \rb{n_{s_\alpha s_\beta, r, s}, \psi, \psi'} = \sum\limits_{\substack{v_4 \ppmod{p^s}\\ (v_4, p^s) = 1}} \sum\limits_{\substack{v_3 \ppmod{p^r}\\ (v_3, p^{r-s}) = 1}}\e\rb{\frac{m_1\ol{v_3}p^s}{p^r}}\e\rb{\frac{m_2 \ol{v_4} v_3^2 + n_2 v_4}{p^s}}.
\ea

Without loss of generality, we assume $\ord_p(m_1) \leq r-s$, and $\ord_p(m_2), \ord_p(n_2) \leq s$. Observe that
\ba
\Kl_p\rb{n_{s_\alpha s_\beta, r,s}, \psi_{m_1, m_2}, \psi_{n_1, n_2}} = p^{k+2l} \Kl_p\rb{n_{s_\alpha s_\beta, r-k-l, s-l} \psi_{m_1 p^{-k}, m_2 p^{-l}}, \psi_{n_1, n_2 p^{-l}}}
\ea
whenever $p^k \mid \rb{m_1, p^{r-s}}$ and $p^l \mid \rb{m_2, n_2, p^s}$. So we may assume $s=0$, $r=s$, or $p\nmid m_1\rb{m_2, n_2}$. 

If $s=0$, then 
\ba
\vb{\Kl\rb{n_{s_\alpha s_\beta, r,0}, \psi, \psi'}} = \bigg|\sum\limits_{\substack{v_3\ppmod{p^r}\\ (v_3, p^r) = 1}}\e\rb{\frac{m_1\ol{v_3}}{p^r}}\bigg| \leq p^{\ord_p(m_1)}. 
\ea

If $r=s$, then
\ba
\vb{\Kl\rb{n_{s_\alpha s_\beta, r,0}, \psi, \psi'}} = \bigg|\sum\limits_{\substack{v_4\ppmod{p^r}\\ (v_4, p) = 1}} \sum\limits_{v_3\ppmod{p^r}}\e\rb{\frac{m_2\ol{v_4}v_3^2 + n_2v_4}{p^r}}\bigg| \leq p^{r+\frac{\ord_p(m_2)}{2} + \frac{\ord_p(n_2)}{2}}
\ea
is just a summation of quadratic Gauss sums, and is easily evaluated. 

Now suppose $p \nmid m_1 \rb{m_2, n_2}$. If $p \mid m_2$ and $s>1$, then
\ba
\Kl\rb{n_{s_\alpha s_\beta, r,s}, \psi, \psi'} = &\sum\limits_{\substack{v_4\ppmod{p^{s-1}}\\ (v_4, p) = 1}} \sum\limits_{\substack{v_3\ppmod{p^r}\\ (v_3, p^{r-s}) = 1}} \sum\limits_{k=0}^{p-1}\e\rb{\frac{m_1\ol{v_3}}{p^{r-s}}}\e\rb{\frac{m_2 \ol{v_4} v_3^2 + n_2 \rb{v_4+ k p^{s-1}}}{p^s}}\\
= &p \sum\limits_{k=0}^{p-1}\e\rb{\frac{n_2k}{p}} \Kl\rb{n_{s_\alpha s_\beta, r-1, s-1}, \psi_{m_1, m_2/p}, \psi'} = 0.
\ea
Now suppose $p \mid m_2$ and $s=1$. We may also assume $r\geq 2$. Then
\ba
\Kl\rb{n_{s_\alpha s_\beta, r,1}, \psi, \psi'} = &\sum\limits_{\substack{v_4\ppmod{p}\\ (v_4,p)=1}} \sum\limits_{\substack{v_3\ppmod{p^r}\\ (v_3,p) = 1}}\e\rb{\frac{m_1\ol{v_3}}{p^{r-1}}}\e\rb{\frac{n_2v_4}{p}} = \begin{cases} p & \text{if } r=2,\\ 0 & \text{if } r>2.\end{cases} 
\ea
The same argument shows that the bound also holds when $p\mid n_2$. Now assume $p\nmid m_1m_2n_2$. When $p$ is odd, the same argument shows that the sum is zero unless $r=2s$; when $p=2$, the sum is zero unless $r=2s$ or $r=2s-1$.

We first consider the case $r=2s$. When $s=1$, we have
\ba
\Kl_p\rb{n_{s_\alpha s_\beta, 2,1}, \psi, \psi'} = \sum\limits_{\substack{v_4\ppmod{p}\\ (v_4, p) = 1}} \sum\limits_{\substack{v_3\ppmod{p^2}\\ (v_3, p) = 1}}\e\rb{\frac{m_1\ol{v_3} + m_2 \ol{v_4} v_3^2 + n_2 v_4}{p}}.
\ea

We apply a theorem of Adolphson and Sperber on exponential sums of Laurent polynomials. Let $k = \F_q$ be a finite field of characteristic $p$. Let 
\ba
f = \sum\limits_{j\in J} a_j x^j \in k[x_1,\cdots, x_n, (x_1 \cdots x_n)^{-1}]
\ea 
be a Laurent polynomial in $n$ variables. We assume that $a_j \neq 0$ for all $j\in J$. Let $\Psi$ be a nontrivial additive character of $k$, we set
\ba
S^*(f) = \sum\limits_{x\in (k^\times)^n} \Psi(f(x)). 
\ea

The Newton polyhedron of $f$, denoted by $\Delta(f)$, is the convex hull in $\R^n$ of the set $J \cup \cb{(0,\cdots, 0)}$. We denote by $V(f)$ the volume of $\Delta(f)$ with respect to the Lebesgue measure on $\R^n$. For a face $\sigma$ (of any dimension) of $\Delta(f)$, we set
\ba
f_\sigma = \sum\limits_{j\in \sigma \cap J} a_j x^j.
\ea
We say that $f$ is non-degenerate with respect to $\Delta(f)$ if for every face $\sigma$ of $\Delta(f)$ that does not contain the origin, the polynomials
\ba
\pd{f_\sigma}{x_1}, \cdots, \pd{f_\sigma}{x_n}
\ea
have no common zeroes in $(\ol k^\times)^n$, where $\ol k$ denotes an algebraic closure of $k$. Then we have the following estimate.
\begin{thm}\label{thm:AS}
\cite[Corollary 4.3]{AS1989} Given an $n$-dimensional integral polyhedron $\Delta$ in $\R^n$, there is a set $\ms S_\Delta$ consisting of all but finitely many prime numbers, such that if $\Char(k) \in \ms S_\Delta$, and
\ba
f \in k[x_1,\cdots, x_n, (x_1\cdots x_n)^{-1}]
\ea
is a non-degenerate Laurent polynomial with $\Delta(f) = \Delta$, then $\vb{S^*(f)} \leq n! V(f) q^{n/2}$. Moreover, when $n=2$, the restriction on $\Char(k)$ can be removed.
\end{thm}

Now set $f(x,y) = \frac{m_1}{x}+\frac{m_2x^2}{y}+n_2y \in \F_p[x,y,(xy)^{-1}]$, and $\Psi:\F_p\to\C^\times$ the standard additive character on $\F_p$. Then $p S^*(f) = \Kl_p(n_{s_\alpha s_\beta,2,1},\psi,\psi')$. We claim that $f$ is non-degenerate whenever $p\neq 2$. The Newton polyhedron $\Delta(f)$ is the triangle with vertices $(x,y) = (-1,0), (2,-1), (0,1)$, and we evaluate $V(f) = 2$. We list the faces $\sigma$ that do not contain the origin, and compute the derivatives $\pd{f_\sigma}{x}, \pd{f_\sigma}{y}$. We denote by $\pb{a_0, \cdots, a_j}$ the $j$-dimensional face of $\Delta(f)$ containing $a_0,\cdots, a_j$. We compute:
\ba
\sigma_1 &= \pb{(-1,0),(2,-1)}, & f_{\sigma_1} &= \frac{m_1}{x}+\frac{m_2x^2}{y}, & \rb{\pd{f_{\sigma_1}}{x}, \pd{f_{\sigma_1}}{y}} &= \rb{-\frac{m_1}{x^2}+\frac{2m_2x}{y}, -\frac{m_2x^2}{y^2}};\\
\sigma_2 &= \pb{(-1,0),(0,1)}, & f_{\sigma_2} &= \frac{m_1}{x}+n_2y, & \rb{\pd{f_{\sigma_2}}{x}, \pd{f_{\sigma_2}}{y}} &= \rb{-\frac{m_1}{x^2}, n_2};\\
\sigma_3 &= \pb{(2,-1),(0,1)}, & f_{\sigma_3} &= \frac{m_2x^2}{y}+n_2y, & \rb{\pd{f_{\sigma_3}}{x}, \pd{f_{\sigma_3}}{y}} &= \rb{\frac{2m_2x}{y}, -\frac{m_2x^2}{y^2}+n_2};\\
\sigma_4 &= \pb{(-1,0)}, & f_{\sigma_4} &= \frac{m_1}{x}, & \rb{\pd{f_{\sigma_4}}{x}, \pd{f_{\sigma_4}}{y}} &=\rb{-\frac{m_1}{x^2},0};\\
\sigma_5 &= \pb{(2,-1)}, & f_{\sigma_5} &= \frac{m_2x^2}{y}, & \rb{\pd{f_{\sigma_5}}{x}, \pd{f_{\sigma_5}}{y}} &=\rb{\frac{2m_2x}{y},-\frac{m_2x^2}{y^2}};\\
\sigma_6 &= \pb{(0,1)}, & f_{\sigma_6} &= n_2y, & \rb{\pd{f_{\sigma_6}}{x}, \pd{f_{\sigma_6}}{y}} &=\rb{0,n_2}.
\ea
Observe that when $p\neq 2$, the terms $-\frac{m_1}{x^2}$, $\frac{2m_2x}{y}$, $-\frac{m_2x^2}{y^2}$, $n_2$ have no zeroes in $(\ol{\F_p}^\times)^2$. Hence we conclude that $f$ is non-degenerate when $p\neq 2$. Now we apply \Cref{thm:AS} and conclude that
\begin{align}\label{eq:ab_S4p}
\vb{S^*(f)} \leq 2V(f)p = 4p. 
\end{align}
for $p\neq 2$. However, by direct computation, the bound \eqref{eq:ab_S4p} also holds for $p=2$. Therefore, for all primes $p$, we have
\ba
\vb{\Kl_p(n_{s_\alpha s_\beta,2,1},\psi,\psi')}\leq 4p^2.
\ea
So the bound holds in this case.

If $s>1$, we apply the stationary phase method, following \cite{DF1997}. Let $V$ be a smooth scheme of dimension $n$, and $f:V\to \A^1 = \A_{\Z_p}^1$ a $\Z_p$-morphism. We consider the exponential sum
\begin{align}\label{eq:stat_phase_def}
S_m(f) := \sum\limits_{x\in V(\Z/p^m\Z)}\e\rb{\frac{f(x)}{p^m}}.
\end{align}
Let $j\leq m$ be a positive integer. We write
\begin{align}\label{eq:approx_crit_def}
D(\Z/p^j\Z) := \cbm{x \in V(\Z/p^j\Z)}{\nabla f(x) \equiv 0\pmod{p^j}}
\end{align}
to denote the ``approximate critical points'' of $f$. For $\ol x \in (\Z/p^j\Z)^n$, we define
\ba
S_m(f)_{\ol x} = \sum\limits_{\substack{x \in V(\Z/p^m\Z)\\ x\equiv \ol x\ppmod{p^j}}}\e\rb{\frac{f(x)}{p^m}}.
\ea
Clearly we have
\ba
S_m(f) = \sum\limits_{\ol x \in (\Z/p^j\Z)^n} S_m(f)_{\ol x}.
\ea

\begin{thm}\label{thm:DF1.8a}
\cite[Theorem 1.8(a)]{DF1997} If $2j\leq m$, then $S_{\ol x} = 0$ unless $\ol x \in D(\Z/p^j\Z)$. Now suppose $m=2j$ or $2j+1$, and let $x\in (\Z/p^m\Z)^n$ map to $\ol x \in D(\Z/p^j\Z)$. If $m=2j$, then we have
\ba
S_m(f)_{\ol x} = p^{mn/2}\e\rb{\frac{f(x)}{p^m}}.
\ea
If $m=2j+1$, then we have
\ba
S_m(f)_{\ol x} = p^{(m-1)n/2}\e\rb{\frac{f(x)}{p^m}} \sum\limits_{y \in (\Z/p\Z)^n}\e\rb{\frac{\frac{1}{2} y^T H_x y + p^{-j} \nabla f(x) \cdot y}{p}},
\ea
where $H_x$ is the Hessian matrix of $f$ at $x$. In particular, if we let $t$ denote the maximum value of $n - \rank_{\F_p} H_{\ol x}$ for $\ol x\in D(\Z/p^j\Z)$, then $\vb{S}\leq \vb{D(\Z/p^j\Z)} p^{(mn+t)/2}$.
\end{thm}

Now we apply the stationary phase method. Let $f(x,y) = \frac{m_1}{x} + \frac{m_2 x^2}{y} + n_2 y$. Consider the sum
\ba
S = \sum\limits_{x, y\in \rb{\Z/p^s\Z}^\times}\e\rb{\frac{f(x,y)}{p^s}} = p^{-s} \Kl_p \rb{n_{s_\alpha s_\beta, 2s, s}, \psi, \psi'}.
\ea
Let $j\geq 1$ be such that $2j\leq s$. Define as in \eqref{eq:approx_crit_def}
\begin{align*}
D\rb{\Z/p^j\Z} &= \cbm{(x,y) \in \rb{\Z/p^j\Z}^\times \times \rb{\Z/p^j\Z}^\times}{\nabla f(x,y) \equiv 0\pmod{p^j}}\\
&= \cbm{\rb{x,y} \in \rb{\Z/p^j\Z}^\times \times \rb{\Z/p^j\Z}^\times}{\begin{array}{l} 2m_2 x^3 \equiv m_1y \pmod{p^j},\\
m_2 x^2 \equiv n_2 y^2 \pmod{p^j}\end{array}}.
\end{align*}

It is straightforward to check that $\vb{D\rb{\Z/p^j\Z}} \leq 4$, and $H_{x,y}$ is invertible over $\F_p$ for all $(x,y)\in D\rb{\Z/p^j\Z}$, so $\rank_{\F_p} H_{x,y} = 2$. So we deduce from \Cref{thm:DF1.8a} that
\ba
\vb{\Kl_p\rb{n_{s_\alpha s_\beta, r,s}, \psi, \psi'}} \leq 4p^{2s}.
\ea

Now it remains to tackle the case $p=2$, $r=2s-1$. As $p$ is fixed, it suffices to prove the bound for sufficiently large $s$, so we can always use the stationary phase method. Let $f(x,y) = \frac{2m_1}{x} + \frac{m_2x^2}{y} + n_2 y$. Consider the sum
\ba
S = \sum\limits_{x,y\in(\Z/p^s\Z)^\times}\e\rb{\frac{f(x,y)}{p^s}} = p^{-s+1} \Kl_p(n_{s_\alpha s_\beta, 2s-1,s},\psi,\psi'). 
\ea
Let $j\geq 1$ be such that $2j\leq s$. Define as in \eqref{eq:approx_crit_def}
\begin{align*}
D\rb{\Z/p^j\Z} &= \cbm{(x,y) \in \rb{\Z/p^j\Z}^\times \times \rb{\Z/p^j\Z}^\times}{\nabla f(x,y) \equiv 0\pmod{p^j}}\\
&= \cbm{\rb{x,y} \in \rb{\Z/p^j\Z}^\times \times \rb{\Z/p^j\Z}^\times}{\begin{array}{l} 2m_2 x^3 \equiv 2m_1y \pmod{p^j},\\
m_2 x^2 \equiv n_2 y^2 \pmod{p^j}\end{array}}.
\end{align*}
Then we have $\vb{D\rb{\Z/p^j\Z}} \leq 16$. The Hessian $H_{x,y}$ is not invertible, but nevertheless we have from \Cref{thm:DF1.8a} that
\ba
\vb{\Kl_p\rb{n_{s_\alpha s_\beta, 2s-1,s}, \psi, \psi'}} \leq 64p^{2s-1}.
\ea
This finishes the proof of the bound for $\Kl_p\rb{n_{s_\alpha s_\beta, r,s}, \psi, \psi'}$.

\subsection{Bound for $\Kl_p\rb{n_{s_\beta s_\alpha, r, s}, \psi, \psi'}$}

As we have mentioned in \Cref{section:Sp4Kloosterman}, the Kloosterman sum $\Kl_p\rb{n_{s_\beta s_\alpha, r, s}, \psi, \psi'}$ differs from the $\GL(3)$ Kloosterman sum $S\rb{n_1, m_1, m_2; p^r, p^{s-r}}$ just by a factor of $p^r$. So our bound immediately follows from the estimate given by Larsen \cite[Appendix]{BFG1988}, whose proof we omit here.

\subsection{Bound for $\Kl_p\rb{n_{s_\alpha s_\beta s_\alpha, r, s}, \psi, \psi'}$}

We make use of the decomposition for Kloosterman sums in \Cref{section:stratification} to obtain a non-trivial bound for $\Kl_p\rb{n_{s_\alpha s_\beta s_\alpha, r, s}, \psi, \psi'}$.

Let $w = s_\alpha s_\beta s_\alpha$, and $n = n_{s_\alpha s_\beta s_\alpha, r, s}$. Note that we have $s\leq 2r$. Then $\Delta_w = \cb{\alpha}$, and
\ba
A_w(\ell) = \rb{\Z/p^\ell\Z}^2 \times \rb{\Z/p^\ell\Z}.
\ea
Let $t = \diag\rb{a_1, a_2, ca_1^{-1}, ca_2^{-1}}\in\mc T$. Then $s = n^{-1}tn = \diag\rb{ca_1^{-1}, a_2, a_1, ca_2^{-1}}$. We compute
\ba
\kappa'_1\rb{t*x} = ca_1^{-1}a_2^{-1} \kappa_1'(x).
\ea
So
\ba
V_w(\ell) = \cbm{(\lambda,\lambda') \in A_w(\ell)^\times}{\lambda_1\lambda_2\lambda'_1= 1}.
\ea
If $\theta:A_w(\ell) \to \C^\times$ is given by
\ba
\theta(\lambda, \lambda') =\e\rb{\frac{n_1\lambda_1+n_2\lambda_2}{p^\ell}}\e\rb{\frac{n'_1\lambda'_1}{p^\ell}}, \quad n_1, n_2, n'_1\in\Z,
\ea
then
\begin{align}\label{eq:abaKloosterman_GL2decomp}
S_w\rb{\theta, \ell} = \sum\limits_{\lambda_2\in\rb{\Z/p^\ell\Z}^\times} \e\rb{\frac{n_2\lambda_2}{p^\ell}} S\rb{n_1\lambda_2^{-1}, n'_1; p^\ell}.
\end{align}

Suppose $x_{a,b}^{v_3} \in X(n)$ has Plücker coordinates
\ba
\rb{v_1, v_2, v_3, v_4; v_{14}} = \rb{p^r, p^{r-a}, v_3, p^{r-b}; p^s}. 
\ea
Let $\delta = \rb{p^{r-a}, p^a v_3 + p^{r-b}}$. Then $v_{14} = \frac{p^{r+a}}{\delta}$. This says $s-r \leq a \leq \frac{s}{2}$, $b\leq r$. Then $\delta = p^{r+a-s}$. Then
\ba
u'\rb{x_{a,b}^{v_3}} = \bp 1 & p^{-a} & v_3p^{-r} & p^{-b}\\ & 1 & p^{-b}\\ &&1\\&&-p^{-a}&1\ep \pmod{ U\rb{\Z_p}}. 
\ea
Let $X_{a,b}^{v_3}(n) = \mc T * x_{a,b}^{v_3}$, and define
\ba
S_{a,b}^{v_3} \rb{n,\psi,\psi'} = \sum\limits_{x \in X_{a,b}^{v_3}(n)} \psi\rb{u(x)} \psi'\rb{u'(x)}.
\ea
We also let
\ba
X_{a,b}(n) = \coprod\limits_{\substack{v_3\pmod{p^r}\\ \rb{p^{r-a}, p^a v_3 + p^{r-b}} = p^{r+a-s}}} X_{a,b}^{v_3}(n),
\ea
and
\ba
S_{a,b}\rb{n,\psi,\psi'} = \sum\limits_{x\in X_{a,b}(n)} \psi\rb{u(x)} \psi'\rb{u'(x)}.
\ea
Let $x\in X(n)$, with Plücker coordinates $v_2 = v_{2,x}$, $v_4 = v_{4,x}$. Then $\ord_p(v_{2,x}) = r-a$, $\ord_p(v_{4,x}) = r-b$ for some $s-r \leq a \leq s/2$, $0\leq b\leq r$. So $x$ lies in the $\mc T$-orbit of $x_{a,b}^{v_3}$ for some $v_3\pmod{p^r}$, and hence $x \in X_{a,b}(n)$. This gives a partition
\ba
X\rb{n} = \coprod\limits_{\substack{s-r\leq a \leq s/2\\ 0\leq b\leq r}} X_{a,b}(n).
\ea
As $r\geq \frac{s}{2} \geq a$, $r\geq b$, we see that $u(x), u'(x)$ have entries in $p^{-r}\Z_p/\Z_p$ for all $x \in X(n)$. Let $\mc S_{a,b}$ be a finite subset of $\Z_p$ such that
\ba
X_{a,b}(n) = \coprod\limits_{v_3\in \mc S_{a,b}} X_{a,b}^{v_3} (n).
\ea

By \Cref{thm:Stevens4.10}, we have
\ba
S_{a,b}\rb{n,\psi,\psi'} = p^{-4r} \rb{1-p^{-1}}^{-2} \sum\limits_{v_3\in \mc S_{a,b}} \vb{X_{a,b}^{v_3}(n)} S_w \rb{\theta_{a,b}^{v_3}; 2r},
\ea
where
\ba
\theta_{a,b}^{v_3} (\lambda,\lambda') =\e\rb{\frac{m_2u\lambda_2}{p^s}}\e\rb{\frac{m_1\hat v_2 \lambda_1 + n_1 p^{r-a} \lambda'_1}{p^r}},
\ea
with $\hat v_2$ and $u$ given as in \eqref{eq:abaKloosterman_v2hat} and \eqref{eq:abaKloosterman_u}. By \eqref{eq:abaKloosterman_GL2decomp}, we have
\begin{align}\label{eq:abaKloosterman_theta_sum}
S_w \rb{\theta_{a,b}^{v_3}; 2r} = \sum\limits_{x, y\in\rb{\Z/p^{2r}\Z}^\times} \e\rb{\frac{m_2u x}{p^s}} \e\rb{\frac{m_1\hat v_2 \ol{x}y + n_1p^{r-a}\ol{y}}{p^r}}.
\end{align}
Since the size of the $\mc T$-orbit of $x_{a,b}^{v_3}$ is bounded by $p^{a+b}$, we have
\begin{align}\label{eq:abaKloosterman_Xab}
\sum\limits_{v_3\in\mc S_{a,b}} \vb{X_{a,b}^{v_3}(n)} \leq \vb{\mc S_{a,b}} p^{a+b} \leq p^{r+a+b}. 
\end{align}

We estimate the size of $S_w \rb{\theta_{a,b}^{v_3}; 2r}$ below. We start by computing the order of $\hat v_2$ and $u$ in \eqref{eq:abaKloosterman_theta_sum}. From \eqref{eq:abaKloosterman_v2hat}, it is clear that $\ord_p\rb{\hat v_2} = s-a$. Now we consider $\ord_p(u)$. If $a\neq \frac{s}{2}$, then we have (after putting $v'_2 = \ol{v'_2} = 1$)
\ba
u = &p^{a+r-s} \rb{-p^av_3 + v_4} + \ol{V'} v_3^2p^{2a}\\
= &p^{a+r-s} \rb{p^av_3 + v_4} - 2 v_3 p^{2a+r-s} + \ol{V'} v_3^2p^{2a}\\
= &p^{2a+2r-2s} V' - 2 v_3 p^{2a+r-s} + \ol{V'} v_3^2p^{2a}\\
= &p^{2a} \ol{V'} \rb{p^{2r-2s} V'^2 - 2p^{r-s}v_3 V' + v_3^2}\\
= &p^{2a} \ol{V'} \rb{p^{r-s}V'-v_3}^2\\
= &p^{2a} \ol{V'} \rb{p^{-a} v_4}^2\\
= &v_4^2 \ol{V'}.
\ea
So $\ord_p(u) = 2\rb{r-b}$. If $a=\frac{s}{2}$, then (again we set $v'_2 = \ol{v'_2} = 1$)
\begin{align}\label{eq:abaKloosterman_u_criterion}
u = &-v_3 p^{2a+r-s} + v_4 p^{a+r-s} = p^{a+r-s} \rb{2v_4 - \rb{p^a v_3 + v_4}}. 
\end{align}
This form will be useful in computing $\ord_p(u)$, when more conditions are given. 

Case I: Suppose $s<r$. We deduce from \eqref{eq:abaKloosterman_v2hat} that $\ord_p(v_3) = 0, \ord_p(v_4) = a$, so only terms with $r=a+b$ contribute. When $a \neq \frac{s}{2}$, we have $\ord_p(u) = 2\rb{r-b} = 2a$. When $a = \frac{s}{2}$, we can still take $\ord_p(u) = s = 2a$. So $\ord_p(u) = 2a$ always holds. 

\ber
\item Suppose $a\leq \frac{2s-r}{3}$. Write $u = p^{2a} u'$. Let 
\ba
t = \min\cb{\ord_p(m_2), \ord_p(m_1)+2s-r-3a, \ord_p(n_1) + s-3a},
\ea
and
\ba
f(x,y) = p^{-t} \rb{m_2 u' y + \frac{m_1 \hat v_2 p^{s-r-2a} x}{y} + \frac{n_1 p^{s-3a}}{x}} = m'_2 y + \frac{m'_1 x}{y} + \frac{n'_1}{x},
\ea
where $m'_1 = m_1 \hat v_2 p^{s-r-2a-t}$, $m'_2 = m_2 u' p^{-t}$, $n'_1 = n_1p^{s-3a-t}$. Consider the sum
\ba
S = \sum\limits_{x,y\in(\Z/p^{s-2a-t}\Z)^\times} e\rb{\frac{f(x,y)}{p^{s-2a-t}}} = p^{2s-4a-4r-2t} S_w\rb{\theta_{a,b}^{v_3}; 2r}. 
\ea

When $s-2a-t>1$, let $j\geq 1$ be such that $2j\leq s-2a-t$. Define as in \eqref{eq:approx_crit_def}
\begin{align*}
D\rb{\Z/p^j\Z} &= \cbm{(x,y) \in \rb{\Z/p^j\Z}^\times \times \rb{\Z/p^j\Z}^\times}{\nabla f(x,y) \equiv 0\pmod{p^j}}\\
&= \cbm{\rb{x,y} \in \rb{\Z/p^j\Z}^\times \times \rb{\Z/p^j\Z}^\times}{\begin{array}{l} m'_1 x^2 \equiv n'_1y \pmod{p^j},\\
m'_2 y^2 \equiv m'_1 x \pmod{p^j}\end{array}}.
\end{align*}
Note that at least one of $m'_1$, $m'_2$ and $n'_1$ is not divisible by $p$. It then follows that $D\rb{\Z/p^j\Z}$ is empty unless $\ord_p(m_2) = \ord_p(m_1)+2s-r-3a = \ord_p(n_1) + s-3a = t$. But then 
\ba
S = p^{4a+2t-2s} \Kl_p(n_{s_\beta s_\alpha, s-2a-t, 3s-6a-3t}, \psi_{m'_1,m'_2}, \psi_{n'_1,0}),
\ea
with $p\nmid m'_1m'_2n'_1$. So it follows from the bound for $\Kl_p(n_{s_\beta s_\alpha,r,s},\psi,\psi')$ that
\begin{align}\label{eq:abaKloosterman_1a_theta}
\vb{S_w\rb{\theta_{a,b}^{v_3}; 2r}} \ll p^{4r+2a-s+t}.
\end{align}

Now suppose $s-2a-t=1$. If $p \nmid m'_1m'_2n'_1$, then it follows by the theorem of Deligne \cite[Sommes. trig., 7.1.3]{Deligne1977} that $S\ll p$. When $p$ divides some (but not all) of $m'_1$, $m'_2$, $n'_1$, then the sum reduces to a Ramanujan sum, and is easily evaluated that $S\ll p$ as well. So the bound \eqref{eq:abaKloosterman_1a_theta} also holds for this case.

\begin{rmk}
\Cref{thm:AS}, itself a generalisation of Deligne's theorem, also applies to give the same bound.
\end{rmk}

The bounds for $S_w\rb{\theta_{a,b}^{v_3}; 2r}$ in other cases are obtained analogously, and we shall omit the repetitive computations thereafter.

\item
Suppose $a>\frac{2s-r}{3}$. Write $\hat v_2 = p^{s-a} \hat v'_2$. Let 
\ba
t = \min\cb{\ord_p(m_2)+r+3a-2s, \ord_p(m_1), \ord_p(n_1)+r-s},
\ea
and
\ba
f(x,y) = p^{-t} \rb{m_2 u p^{r+a-2s} y+\frac{m_1 \hat v'_2 x}{y} + \frac{n_1 p^{r-s}}{x}} = m'_2 y + \frac{m'_1 x}{y} + \frac{n'_1}{x},
\ea
where $m'_1 \hat v'_2 p^{-t}$, $m'_2 = m_2 u p^{r+a-2s-t}$, $n'_1 = n_1 p^{r-s-t}$. Then we have
\ba
S = \sum\limits_{x,y\in (\Z/p^{r+a-s-t}\Z)^\times} e\rb{\frac{f(x,y)}{p^{r+a-s-t}}} = p^{2a-2r-2s-2t} S_w\rb{\theta_{a,b}^{v_3}; 2r}.
\ea
Then we obtain analogously
\ba
\vb{S_w\rb{\theta_{a,b}^{v_3}; 2r}} \ll p^{3r-a+s+t}. 
\ea
\ee

Note that we have $\rb{p^{r-a}, p^a \rb{v_3+1}} = p^{r+a-s}$. A necessary condition for this to hold is that $p^{r-s} \mid v_3+1$. So $\vb{\mc S_{a,b}} \leq p^s$. So, from \eqref{eq:abaKloosterman_Xab} we actually have
\ba
\sum\limits_{v_3\in \mc S_{a,b}} \vb{X_{a,b}^{v_3}(n)} \leq p^{s+a+b}. 
\ea
Hence
\ba
\vb{\Kl_p\rb{n,\psi,\psi'}} &\leq \sum\limits_{\substack{0\leq a \leq s/2\\ b = r-a}} \vb{S_{a,b}\rb{n,\psi,\psi'}}\\
&\ll \sum\limits_{\substack{0\leq a \leq s/2\\ b = r-a}} p^{-4r} p^{s+a+b} S_w\rb{\theta_{a,b}^{v_3};2r}\\
&\ll \sum\limits_{\substack{0\leq a \leq s/2}} \min\cb{p^{r+2a+\ord_p(m_2)}, p^{s-a+\min\cb{s+\ord_p(m_1), r+\ord_p(n_1)}}}\\
&\ll p^{\frac{r}{3} + \frac{2s}{3} + \frac{2}{3} \min\cb{\ord_p(m_1)+s, \ord_p(n_1)+r} + \frac{1}{3} \ord_p(m_2)}. 
\ea

Case II: Suppose $s=r$. We deduce from \eqref{eq:abaKloosterman_v2hat} that when $a\neq 0$, then $\ord_p(v_3) = 0, \ord_p(v_4) \geq a$. So, only terms with $r\geq a+b$ contribute. When $a\neq \frac{s}{2}$, we have $\ord_p(u) = 2\rb{r-b}$. When $a=\frac{s}{2}$, we still have $\ord_p(u) \leq s = 2\rb{r-b}$. So $\ord_p(u) \leq 2\rb{r-b}$ always holds. We compute
\ba
\vb{S_w\rb{\theta_{a,b}^{v_3}; 2r}} \ll p^{2r} \min\cb{p^{3r-2b+\ord_p(m_2)}, p^{2r-a+\min\cb{\ord_p(m_1), \ord_p(n_1)}}}. 
\ea
Hence
\ba
\vb{\Kl_p\rb{n,\psi,\psi'}} &\leq \sum\limits_{\substack{0\leq a \leq r/2\\ b \leq r-a}} \vb{S_{a,b}\rb{n,\psi,\psi'}}\\
&\ll \sum\limits_{\substack{0\leq a \leq s/2\\ b \leq r-a}} p^{-4r} p^{r+a+b} \rb{p^{2r} \min\cb{p^{3r-2b+\ord_p(m_2)}, p^{2r-a+\min\cb{\ord_p(m_1), \ord_p(n_1)}}}}\\
&\ll \sum\limits_{\substack{0\leq a \leq s/2\\ b \leq r-a}} p^{-r+a+b} \min\cb{p^{3r-2b+\ord_p(m_2)}, p^{2r-a+\min\cb{\ord_p(m_1), \ord_p(n_1)}}}\\
&\ll p^{\frac{5r}{3} + \frac{2}{3}\min\cb{\ord_p(m_1), \ord_p(n_1)} + \frac{1}{3} \ord_p(m_2)}. 
\ea

Case III: $2r>s>r$. We consider the following subcases:
\bea
\item Suppose $a=s-r$. Then the condition $\rb{p^{r-a}, p^a v_3 + p^{r-b}} = 1$ implies $b=r$. So $\ord_p(u) = 0$. We deduce from \eqref{eq:abaKloosterman_v2hat} that $\hat v_2 = 0$. So 
\ba
\vb{S_w\rb{\theta_{a,b}^{v_3}; 2r}} \ll p^{3r-s} \min\cb{p^{r+\ord_p(m_2)}, p^{2r+\ord_p(n_1)}}.
\ea

\item Suppose $s-r < a < \frac{s}{2}$. Then we deduce from \eqref{eq:abaKloosterman_v2hat} that $\ord_p(v_3) = 0$, $\ord_p(v_4) \geq a$. So $a+b\leq r$. Meanwhile, as $r+a-s < a$, the condition $\rb{p^{r-a}, p^a v_3 + p^{r-b}} = p^{r+a-s}$ says $r-b = r+a-s$, which implies $a+b = s >r$, a contradiction. So there is no contribution from this case.

\item Suppose $a = \frac{s}{2}$. Again, we deduce from \eqref{eq:abaKloosterman_v2hat} that $\ord_p(v_3) = 0$, $\ord_p(v_4) \geq a$. So, only terms with $r\geq a+b$ contribute. In this case, we don't have a good bound for $\ord_p(u)$. So
\ba
\vb{S_w\rb{\theta_{a,b}^{v_3}; 2r}} \ll p^{3r+\min\cb{\frac{s}{2}+\ord_p(m_1), r-\frac{s}{2}+\ord_p(n_1)}}.
\ea
\ee 
Hence
\ba
\vb{\Kl_p\rb{n,\psi,\psi'}} \leq &\sum\limits_{\substack{s-r\leq a \leq s/2\\ b \leq r-a}} \vb{S_{a,b}\rb{n,\psi,\psi'}}\\
\ll &\sum\limits_{\substack{a=s-r\\ b=r}} p^{-4r} p^{r+a+b} \rb{p^{3r-s} \min\cb{p^{r+\ord_p(m_2)}, p^{2r+\ord_p(n_1)}}}\\
&+\sum\limits_{\substack{a=s/2\\ b\leq r-s/2}} p^{-4r} p^{r+a+b} \rb{p^{3r+\min\cb{\frac{s}{2}+\ord_p(m_1), r-\frac{s}{2}+\ord_p(n_1)}}}\\
\ll &p^{r+\min\cb{\ord_p(m_2), r+\ord_p(n_1)}} + p^{r+\min\cb{\frac{s}{2}+\ord_p(m_1), r-\frac{s}{2}+\ord_p(n_1)}}.
\ea

Case IV: $s = 2r$. In this case, we have $a = r$, and $v_3, v_4 = p^{r-b}$ is arbitrary. We deduce from \eqref{eq:abaKloosterman_v2hat} that $\hat v_2 = 0$. We consider the following subcases:
\bea
\item Suppose $b=0$. We may assume $v_4=0$. Then $\ord_p(u) = r+\ord_p(v_3)$. We compute
\ba
\vb{S_w\rb{\theta_{a,b}^{v_3}; 2r}} \ll p^r \min\cb{p^{2r+\ord_p(v_3)+\ord_p(m_2)}, p^{2r+\ord_p(n_1)}}.
\ea
Fix $c\leq r$. Then
\ba
\vb{\cbm{v_3\in\mc S_{a,b}}{\ord_p(v_3) = c}} \leq p^{r-c}.
\ea
\item Suppose $b>0$. Then $\ord_p(u) = r-b$. We compute
\ba
\vb{S_w\rb{\theta_{a,b}^{v_3}; 2r}} \ll p^r \min\cb{p^{2r-b+\ord_p(m_2)}, p^{2r+\ord_p(n_1)}}.
\ea
\ee
Hence
\ba
\vb{\Kl_p\rb{n,\psi,\psi'}} \leq &\sum\limits_{\substack{a = r/2\\ b\leq r}} \vb{S_{a,b}\rb{n,\psi,\psi'}}\\
\ll &\sum\limits_{\substack{a=r/2\\ b=0\\ c\leq r}} p^{-4r} p^{r-c+a+b} \rb{p^r \min\cb{p^{2r+c+\ord_p(m_2)}, p^{2r+\ord_p(n_1)}}}\\
&+\sum\limits_{\substack{a=r/2\\ b>0}} p^{-4r} p^{r+a+b} \rb{p^r \min\cb{p^{2r-b+\ord_p(m_2)}, p^{2r+\ord_p(n_1)}}}\\
\ll &p^{r+\min\cb{\ord_p(m_2), r+\ord_p(n_1)}}.
\ea
This finishes the proof of the bound for $\Kl_p\rb{n_{s_\alpha s_\beta s_\alpha,r,s}, \psi, \psi'}$.

\subsection{Bound for $\Kl_p\rb{n_{s_\beta s_\alpha s_\beta, r, s}, \psi, \psi'}$}

We make use of the decomposition for Kloosterman sums in \Cref{section:stratification} to obtain a non-trivial bound for $\Kl_p\rb{n_{s_\beta s_\alpha s_\beta, r, s}, \psi, \psi'}$.

Let $w = s_\beta s_\alpha s_\beta$, and $n = n_{s_\beta s_\alpha s_\beta, r, s}$. Note that we have $r\leq s$. Then $\Delta_w = \cb{\beta}$, and
\ba
A_w\rb{\ell} = \rb{\Z/p^\ell\Z}^2 \times \rb{\Z/p^\ell\Z}.
\ea
Let $t = \diag \rb{a_1, a_2, ca_1^{-1}, ca_2^{-1}} \in \mc T$. Then $s = n^{-1}tn = \diag \rb{ca_2^{-1}, ca_1^{-1}, a_2, a_1}$. We compute
\ba
\kappa'_2 \rb{t * x} = ca_1^{-2} \kappa'_2(x).
\ea
So
\ba
V_w(\ell) = \cbm{(\lambda,\lambda') \in A_w(\ell)^\times}{\lambda_1^2\lambda_2\lambda'_2 = 1}.
\ea
If $\theta: A_w(\ell) \to \C^\times$ is given by
\ba
\theta(\lambda,\lambda') = \e\rb{\frac{n_1\lambda_1+n_2\lambda_2}{p^\ell}} \e\rb{\frac{n'_2\lambda'_2}{p^\ell}}, \quad n_1, n_2, n'_2\in\Z,
\ea
then
\begin{align}\label{eq:babKloosterman_GL2decomp}
S_w\rb{\theta, \ell} = \sum\limits_{\lambda_1\in\rb{\Z/p^\ell\Z}^\times} \e\rb{\frac{n_1\lambda_1}{p^\ell}} S \rb{n_2 \lambda_1^{-2}, n'_2; p^\ell}.
\end{align}

Suppose $x_{a,b}^{v_{23}} \in X(n)$ has Plücker coordinates
\ba
\rb{v_{12}, v_{13}, v_{14}, v_{23}} = \rb{p^s, p^{s-a}, p^{s-b}, v_{23}}. 
\ea
The condition $(v_{12}, v_{14}) \mid v_{13}^2$ says $s-b \leq 2\rb{s-a}$, that is, $2a-b\leq s$. We also have $\max\cb{a,b} = r$. Then
\ba
u'\rb{x_{a,b}^{v_{23}}} = \bp 1 & & -v_{23}p^{-s} & p^{-a}\\ &1&p^{-a}&p^{-b}\\&&1\\&&&1\ep \pmod{U\rb{\Z_p}}.
\ea
Let $X_{a,b}^{v_{23}} (n) = \mc T * x_{a,b}^{v_{23}}$, and define
\ba
S_{a,b}^{v_{23}} \rb{n, \psi, \psi'} = \sum\limits_{x\in X_{a,b}^{v_{23}} (n)} \psi\rb{u(x)} \psi'\rb{u'(x)}.
\ea
We also let
\ba
X_{a,b} (n) = \coprod\limits_{\substack{v_{23} \pmod{p^s}\\ \rb{p^{s-r}, v_{23}, p^{-b}v_{23}-p^{s-2a}}=1}} X_{a,b}^{v_{23}} (n),
\ea
and
\ba
S_{a,b}\rb{n, \psi, \psi'} = \sum\limits_{x\in X_{a,b}(n)} \psi\rb{u(x)} \psi'\rb{u'(x)}. 
\ea
Again we have a partition
\ba
X(n) = \coprod\limits_{\substack{0\leq a, b \leq r\\ \max\cb{a,b} = r\\ 2a-b \leq s}} X_{a,b} (n). 
\ea
It is clear that $u(x), u'(x)$ have entries in $p^{-s}\Z_p/\Z_p$ for all $x\in X(n)$. Let $\mc S_{a,b}$ be a finite subset of $\Z_p$ such that 
\ba
X_{a,b}(n) = \coprod\limits_{v_{23} \in \mc S_{a,b}} X_{a,b}^{v_{23}} (n).
\ea

By \Cref{thm:Stevens4.10}, we have
\ba
S_{a,b} \rb{n, \psi, \psi'} = p^{-2s} \rb{1-p^{-1}}^{-2} \sum\limits_{v_{23}\in \mc S_{a,b}} \vb{X_{a,b}^{v_{23}} (n)} S_w \rb{\theta_{a,b}^{v_{23}}; s},
\ea
where
\ba
\theta_{a,b}^{v_{23}} (\lambda,\lambda') = \e\rb{\frac{m_1 u \lambda_1}{p^r}} \e\rb{\frac{m_2 \hat v_{14} \lambda_2 + n_2 p^{s-b} \lambda'_2}{p^s}}. 
\ea
with $\hat v_{14}$ and $u$ given as in \eqref{eq:babKloosterman_u} and \eqref{eq:babKloosterman_v14hat}. By \eqref{eq:babKloosterman_GL2decomp}, we have
\begin{align}\label{eq:babKloosterman_theta_sum}
S_w\rb{\theta_{a,b}^{v_{23}}; s} = \sum\limits_{x,y \in \rb{\Z/p^s\Z}^\times} \e\rb{\frac{m_1u\ol{x}}{p^r}} \e\rb{\frac{m_2 \hat v_{14} x^2 \ol{y} + n_2 p^{s-b} y}{p^s}}.
\end{align}
Since the size of the $\mc T$-orbit of $x_{a,b}^{v_{23}}$ is bounded by $p^{a+b}$, we have
\begin{align}\label{eq:babKloosterman_Xab}
\sum\limits_{v_{23}\in\mc S_{a,b}} \vb{X_{a,b}^{v_{23}} (n)} \leq \vb{\mc S_{a,b}} p^{a+b} \leq p^{s+a}.
\end{align}

We estimate the size of $S_w \rb{\theta_{a,b}^{v_{23}}; s}$. We start by computing the order of $\hat v_{14}$ and $u$ in \eqref{eq:babKloosterman_theta_sum}. From \eqref{eq:babKloosterman_u}, we see that
\begin{align}\label{eq:babKloosterman_u_criterion}
u p^{r-a} &\equiv v_{23} \pmod{p^r}, & u p^{r-b} &\equiv -p^{s-a} \pmod{p^r}. 
\end{align}
So, if $a=r$, then $u\equiv v_{23}\pmod{p^r}$, and if $b=r$, then $u\equiv -p^{s-a}\pmod{p^r}$. (Recall that $\max\cb{a,b}=r$.) Also, we know that 
\begin{align}\label{eq:babKloosterman_v23criterion}
v_{23} = -p^{s-2a+b} + \beta p^b
\end{align}
for some $\beta\in \Z$ such that $\rb{\beta, p^{s-2r+b}} = 1$ (see \cite[Section 3.2]{Man2020}). Meanwhile, from \eqref{eq:babKloosterman_v14hat}, we see that unless $r=s$, we have $\ord_p\rb{\hat v_{14}} = 2r-b$.

Case I: Suppose $r<\frac{s}{2}$. We deduce from \eqref{eq:babKloosterman_v23criterion} that $\ord_p(v_{23}) = b$. From \eqref{eq:babKloosterman_u_criterion}, we deduce $a\geq b$. So we actually have $a=r$, and then $\ord_p(u) = b$. 

\ber
\item Suppose $b\leq \frac{3r-s}{2}$. Write $u = p^b u'$. Let 
\ba
t = \min\cb{\ord_p(m_1), \ord_p(m_2)+3r-2b-s, \ord_p(n_2)+r-2b}
\ea
and
\ba
f(x,y) = p^{-t} \rb{\frac{m_1u'}{x} + \frac{m_2 \hat v_{14} p^{r-b-s} x^2}{y} + n_2 p^{r-2b}y} = \frac{m'_1}{x} + \frac{m'_2 x^2}{y} + n'_2 y,
\ea
where $m'_1 = m_1 u' p^{-t}$, $m'_2 = m_2 \hat v_{14} p^{r-b-s-t}$, $n'_2 = n_2 p^{r-2b-t}$. Consider the sum
\ba
S = \sum\limits_{x,y\in(\Z/p^{r-b-t}\Z)^\times} e\rb{\frac{f(x,y)}{p^{r-b-t}}} = p^{2r-2s-2b-2t} S_w\rb{\theta_{a,b}^{v_{23}}; s}.
\ea
When $r-b-t>1$, let $j\geq 1$ be such that $2j\leq r-b-t$. Define as in \eqref{eq:approx_crit_def}
\ba
D(\Z/p^j\Z) = &\cbm{(x,y)\in(\Z/p^j\Z)^\times \times (\Z/p^j\Z)^\times}{\nabla f(x,y) \equiv 0\pmod{p^j}}\\
= &\cbm{(x,y)\in(\Z/p^j\Z)^\times \times (\Z/p^j\Z)^\times}{\begin{array}{l} 2m'_2 x^3 \equiv m'_1 y \pmod{p^j}\\ m'_2 x^2 \equiv n'_2 y^2 \pmod{p^j}\end{array}}.
\ea
Note that at least one of $m'_1$, $m'_2$ and $n'_2$ is not divisible by $p$. It then follows that when $p$ is odd, $D(\Z/p^j\Z)$ is empty unless $\ord_p(m_1) = \ord_p(m_2)+3r-2b-s = \ord_p(n_2)+r-2b = t$. But then
\ba
S = p^{t+b-r} \Kl_p(n_{s_\alpha s_\beta, 2r-2b-2t, r-b-t}, \psi_{m'_1,m'_2}, \psi_{0,n'_2}),
\ea
with $p\nmid m'_1m'_2n'_2$. So it follows from the bound for $\Kl_p(n_{s_\alpha s_\beta,r,s},\psi,\psi')$ that
\begin{align}\label{eq:babKloosterman_1a_theta}
\vb{S_w\rb{\theta_{a,b}^{v_{23}};s}} \ll p^{2s-r+b+t}.
\end{align}
When $p=2$, $D(\Z/p^j\Z)$ is empty unless $\ord_p(m_1)-1 = \ord_p(m_2)+3r-2b-s = \ord_p(n_2)+r-2b = t$. Then
\ba
S = p^{t+b-r+1} \Kl_p(n_{s_\alpha s_\beta, 2r-2b-2t-1, r-b-t}, \psi_{m'_1/2,m'_2}, \psi_{0,n'_2}),
\ea
with $p \nmid (m'_1/2)m'_2n'_2$. Again, from the bound for $\Kl_p(n_{s_\alpha s_\beta,r,s},\psi,\psi')$, we see that \eqref{eq:babKloosterman_1a_theta} also holds for this case.

Now suppose $r-b-t=1$. If $p\nmid m'_1 m'_2 n'_1$, then it again follows from \Cref{thm:AS} that $\vb{S}\ll p$. When $p$ divides some (but not all) of $m'_1, m'_2, n'_1$, then the sum reduces to Gauß sums or Ramanujan sums, and is easily evaluated that $\vb{S}\ll p$ as well. So the bound \eqref{eq:babKloosterman_1a_theta} also holds for this case.

The bounds for $S_w\rb{\theta_{a,b}^{v_{23}}; s}$ in other cases are obtained analogously, and we shall omit the repetitive computations thereafter.

\item Suppose $b> \frac{3r-s}{2}$. Write $\hat v_{14} = p^{2r-b} \hat v'_{14}$. Let
\ba
t = \min\cb{\ord_p(m_1)+s+2b-3r, \ord_p(m_2), \ord_p(n_2)+s-2r},
\ea
and
\ba
f(x,y) = p^{-t} \rb{\frac{m_1up^{s+b-3r}}{x} + \frac{m_2 \hat v'_{14} x^2}{y} + n_2 p^{s-2r} y} = \frac{m'_1}{x} + \frac{m'_2 x^2}{y} + n'_2 y,
\ea
where $m'_1 = m_1 u p^{s+b-3r-t}$, $m'_2 = m_2 \hat v'_{14} p^{-t}$, $n'_2 = n_2 p^{s-2r-t}$. Then we have 
\ba
S = \sum\limits_{x,y\in(\Z/p^{s+b-2r-t}\Z)^\times} e\rb{\frac{f(x,y)}{p^{s+b-2r-t}}} = p^{2b-4r-2t} S_w\rb{\theta_{a,b}^{v_{23}}; s}.
\ea
Then we obtain analogously
\ba
\vb{S_w\rb{\theta_{a,b}^{v_{23}}; s}} \ll p^{s+2r-b+t}. 
\ea
\ee
Hence
\ba
\vb{\Kl_p\rb{n, \psi, \psi'}} &\leq \sum\limits_{\substack{a=r\\ 0\leq b\leq r}} \vb{S_{a,b}\rb{n,\psi,\psi'}}\\
&\ll \sum\limits_{\substack{a=r\\ 0\leq b\leq r}} p^{-2s} p^{s+a} \vb{S_w\rb{\theta_{a,b}^{v_{23}}; s}}\\
&\ll \sum\limits_{\substack{a=r\\ 0\leq b\leq r}} p^{-2s} p^{s+a} \rb{p^{s-r} \min\cb{p^{s+b+\ord_p(m_1)}, p^{r-b+\min\cb{2r+\ord_p(m_2), s+\ord_p(n_2)}}}}\\
&\ll p^{\frac{s}{2}+\frac{r}{2}+\frac{1}{2}\min\cb{2r+\ord_p(m_2), s+\ord_p(n_2)}+\frac{1}{2}\ord_p(m_1)}. 
\ea

Case II: Suppose $r=\frac{s}{2}$. We consider the following subcases:
\bea
\item Suppose $b=r$. From \eqref{eq:babKloosterman_u_criterion}, we may assume $u=0$. We compute
\ba
\vb{S_w\rb{\theta_{a,b}^{v_{23}}; s}} \ll p^{\frac{3s}{2}+\min\cb{\ord_p(m_2), \ord_p(n_2)}}. 
\ea

\item Suppose $b<r$. Then $a=r$. From \eqref{eq:babKloosterman_v23criterion}, we see that $v_{23} = \rb{\beta-1} p^b$ for some $\beta\in\Z$ such that $\rb{\beta, p^b} = 1$. So $\ord_p(v_{23}) \geq b$. And from \eqref{eq:babKloosterman_u_criterion}, we deduce that $\ord_p(u) = \ord_p(v_{23})$. We compute
\ba
\vb{S_w\rb{\theta_{a,b}^{v_{23}}; s}} \ll p^{s/2} \min\cb{p^{s+\ord_p(v_{23}) + \ord_p(m_1)}, p^{\frac{3s}{2}-b+\min\cb{\ord_p(m_2), \ord_p(n_2)}}}. 
\ea
\ee

Fix $c\geq b$. Then
\ba
\vb{\cbm{v_{23} \in \mc S_{a,b}}{\ord_p(v_{23}) = c}} \leq p^{s-c}. 
\ea
Hence
\begin{equation*}
\scalebox{0.98}{$
\begin{aligned}
\vb{\Kl_p\rb{n, \psi, \psi'}} \leq &\sum\limits_{\substack{a,b\leq r\\ \max\cb{a,b}=r}} \vb{S_{a,b}\rb{n, \psi, \psi'}}\\
\ll &\sum\limits_{\substack{b=r\\ a\leq r}} p^{-2s} p^{s+a} \rb{p^{\frac{3s}{2}+\min\cb{\ord_p(m_2), \ord_p(n_2)}}}\\
&+\sum\limits_{\substack{a=r\\ b<r\\ b\leq c\leq r}} p^{-2s} p^{s-c+a+b} \rb{p^{s/2} \min\cb{p^{s+\ord_p(v_{23}) + \ord_p(m_1)}, p^{\frac{3s}{2}-b+\min\cb{\ord_p(m_2), \ord_p(n_2)}}}}\\
\ll &p^{\frac{5s}{4} + \frac{1}{2}\ord_p(m_1) + \frac{1}{2}\min\cb{\ord_p(m_2), \ord_p(n_2)}}.
\end{aligned}$}
\end{equation*}

Case III: Suppose $s>r>\frac{s}{2}$. We consider the following subcases:
\bea
\item Suppose $b=r$. Then $\ord_p(u) = s-a$, and $\ord_p(\hat v_{14}) = r$. We compute
\ba
\vb{S_w\rb{\theta_{a,b}^{v_{23}}; s}} \ll p^{s-r} \min\cb{p^{2s-a+\ord_p(m_1)}, p^{r+\min\cb{r+\ord_p(m_2)}, s-r+\ord_p(n_2)}}.
\ea
\item Suppose $b<r$. Then $a=r$. Then from \eqref{eq:babKloosterman_v23criterion} we deduce that $\ord_p(v_{23}) = p^{s-2r+b}$, and hence $\ord_p(u) = p^{s-2r+b}$. We compute
\ba
\vb{S_w\rb{\theta_{a,b}^{v_{23}}; s}} \ll p^{s-r} \min\cb{p^{2s-2r+b+\ord_p(m_1)}, p^{r-b+\min\cb{2r+\ord_p(m_2), s+\ord_p(n_2)}}}.
\ea
\ee
Hence
\begin{equation*}
\scalebox{0.98}{$
\begin{aligned}
\vb{\Kl_p\rb{n, \psi, \psi'}} \leq &\sum\limits_{\substack{a,b\leq r\\ \max\cb{a,b}=r\\ 2a-b\leq s}} \vb{S_{a,b}\rb{n, \psi, \psi'}}\\
\ll &\sum\limits_{\substack{b=r\\ a\leq r}} p^{-2s} p^{s+a} \rb{p^{s-r} \min\cb{p^{2s-a+\ord_p(m_1)}, p^{r+\min\cb{r+\ord_p(m_2)}, s-r+\ord_p(n_2)}}}\\
&+\sum\limits_{\substack{a=r\\ 2r-s\leq b < r}} p^{-2s} p^{s+a} \rb{p^{s-r} \min\cb{p^{2s-2r+b+\ord_p(m_1)}, p^{r-b+\min\cb{2r+\ord_p(m_2), s+\ord_p(n_2)}}}}\\
&\ll p^{s-\frac{r}{2}+\frac{1}{2}\ord_p(m_1)+\frac{1}{2}\min\cb{2r+\ord_p(m_2), s+\ord_p(n_2)}}. 
\end{aligned}$}
\end{equation*}

Case IV: $r=s$. In this case we only have to consider terms with $b=r$. Indeed, if $b<r$, then $a=r$, and then by \eqref{eq:babKloosterman_u_criterion}, we see that $u p^{r-b} \equiv -1\pmod{p^r}$, which says $b=r$, a contradiction. When $b=r$, we have $\ord_p(u) = s-a$, and from \eqref{eq:babKloosterman_v14hat} we may assume $\hat v_{14} = 0$. We compute
\ba
\vb{S_w\rb{\theta_{a,b}^{v_{23}}; s}} \ll \min\cb{p^{2s-a+\ord_p(m_1)}, p^{s+\ord_p(n_2)}}.
\ea
Hence
\ba
\vb{\Kl_p\rb{n, \psi, \psi'}} \leq &\sum\limits_{\substack{b=s\\ a\leq s}} \vb{S_{a,b}\rb{n, \psi, \psi'}}\\
\ll &\sum\limits_{\substack{b=s\\ a\leq s}} p^{-2s} p^{s+a} \rb{\min\cb{p^{2s-a+\ord_p(m_1)}, p^{s+\ord_p(n_2)}}}\\
\ll &p^{s+\min\cb{\ord_p(m_1), \ord_p(n_2)}}.
\ea
This finishes the proof of the bound for $\Kl_p\rb{n_{s_\beta s_\alpha s_\beta,r,s}, \psi, \psi'}$.

\subsection{Bounds for $\Kl_p\rb{n_{w_0, r, s}, \psi, \psi'}$} We show that under the stratification introduced in \Cref{section:stratification}, $\Kl_p\rb{n_{w_0, r, s}, \psi, \psi'}$ decomposes into a sum of products of $\GL(2)$ Kloosterman sums. So the Kloosterman sum can be bounded using \eqref{eq:Kloosterman_GL2bound}.

Let $w = w_0$, and $n = n_{w_0, s, r}$. Then $\Delta_{w_0} = \Delta$, and 
\ba
A_{w_0}(\ell) = \rb{\Z/p^\ell\Z}^2 \times \rb{\Z/p^\ell\Z}^2.
\ea
Let $t = \diag\rb{a_1, a_2, ca_1^{-1}, ca_2^{-1}} \in \mc T$. Then $s = n^{-1}tn = \diag\rb{ca_1^{-1}, ca_2^{-1}, a_1, a_2}$. We compute
\ba
\kappa'_1(t*x) &= a_2a_1^{-1} \kappa'_1(x), & \kappa'_2(t*x) &= c a_2^{-2} \kappa'_2(x).
\ea
So
\ba
V_{w_0}(\ell) = \cbm{(\lambda,\lambda') \in A_{w_0}(\ell)^\times}{ \lambda_1\lambda'_1 = 1, \lambda_2\lambda'_2 = 1}.
\ea
If $\theta: A_{w_0}(\ell) \to \C^\times$ is given by
\ba
\theta(\lambda,\lambda') &= \prod\limits_{i=1}^2 \e\rb{\frac{n_i\lambda_i}{p^\ell}} \prod\limits_{i=1}^2 \e\rb{\frac{n'_i\lambda'_i}{p^\ell}}, & &n_1, n_2, n'_1, n'_2\in\Z,
\ea
then
\begin{align}\label{eq:Kloosterman_GL2decomp}
S_{w_0}\rb{\theta; \ell} = S \rb{n_1, n'_1; p^\ell} S \rb{n_2, n'_2; p^\ell}.
\end{align}

Suppose $x_{a,b}^{v_3, v_4, v_{13}}\in X(n)$ has Plücker coordinates
\ba
\rb{v_1, v_2, v_3, v_4; v_{12}, v_{13}, v_{14}} = \rb{p^r, p^{r-a}, v_3, v_4; p^s, v_{13}, p^{s-b}}. 
\ea
Note that this also says $r\geq a, s\geq b$. Then
\ba
u'\rb{x_{a,b}^{v_3, v_4, v_{13}}} = \bp 1&p^{-a} & v_3p^{-r} & v_4p^{-r}\\ &1&v_{13}p^{-s} & p^{-b}\\ &&1\\&&-p^{-a}&1\ep \pmod{U\rb{\Z_p}}. 
\ea
Let $X_{a,b}^{v_3, v_4, v_{13}}(n) = \mc T * x_{a,b}^{v_3, v_4, v_{13}}$, and define 
\ba
S_{a,b}^{v_3, v_4, v_{13}} \rb{n, \psi, \psi'} = \sum\limits_{x \in X_{a,b}^{v_3, v_4, v_{13}}(n)} \psi\rb{u(x)} \psi'\rb{u'(x)}. 
\ea
We also let
\ba
X_{a,b} (n) = \coprod\limits_{\substack{v_3, v_4 \ppmod{p^r}\\v_{13}\ppmod{p^s}\\\text{conditions}}} X_{a,b}^{v_3, v_4, v_{13}} (n),
\ea
and
\ba
S_{a,b} \rb{n, \psi, \psi'} = \sum\limits_{x \in X_{a,b}(n)} \psi\rb{u(x)} \psi'\rb{u'(x)}.
\ea
We have a partition
\ba
X(n) = \coprod\limits_{\substack{0\leq a\leq r\\ 0\leq b\leq s}} X_{a,b}(n). 
\ea
Now we consider cases $r\geq s$ and $r<s$ separately. 

\ber
\item Suppose $r> s$. As $r\geq a, r\geq s\geq b$, we see that $u(x), u'(x)$ have entries in $p^{-r}\Z_p/\Z_p$ for all $x\in X(n)$. Let $\mc S_{a,b}$ be a finite subset of $\Z_p^3$ such that 
\ba
X_{a,b}(n) = \coprod\limits_{(v_3, v_4, v_{13}) \in \mc S_{a,b}} X_{a,b}^{v_3, v_4, v_{13}} (n).
\ea 

By \Cref{thm:Stevens4.10}, we have
\ba
S_{a,b} \rb{n, \psi, \psi'} = p^{-2r} \rb{1-p^{-1}}^{-2} \sum\limits_{(v_3, v_4, v_{13}) \in \mc S_{a,b}} \vb{X_{a,b}^{v_3, v_4, v_{13}} (n)} S_{w_0} \rb{\theta_{a,b}^{v_3, v_4, v_{13}}; r},
\ea
where
\ba
\theta_{a,b}^{v_3, v_4, v_{13}} (\lambda,\lambda') = \e\rb{\frac{m_1\hat v_2\lambda_1 + n_1p^{r-a}\lambda'_1}{p^r}} \e\rb{\frac{m_2\hat v_{14} + n_2 p^{s-b}}{p^s}}. 
\ea
By \eqref{eq:Kloosterman_GL2decomp}, we have
\ba
S_{w_0} \rb{\theta_{a,b}^{v_3, v_4, v_{13}}; r} = S\rb{m_1\hat v_2, n_1\hat p^{r-a}; p^r} S\rb{m_2\hat v_{14} p^{r-s}, n_2p^{r-b}; p^r}. 
\ea
And we obtain a bound by applying \eqref{eq:Kloosterman_GL2bound}:
\ba
\vb{S_{w_0}\rb{\theta_{a,b}^{v_3, v_4, v_{13}}; r}} \leq 4 p^r \rb{\gcd\rb{m_1\hat v_2, n_1p^{r-a}, p^r} \gcd\rb{m_2\hat v_{14} p^{r-s}, n_2p^{r-b}, p^r}}^{1/2}. 
\ea

\item Suppose $s\geq r$. Then $u(x), u'(x)$ has entries in $p^{-s}\Z_p/\Z_p$ for all $x\in X(n)$. Again, by \Cref{thm:Stevens4.10} we have
\ba
S_{a,b}\rb{n,\psi,\psi'} = p^{-2s} \rb{1-p^{-1}}^{-2} \sum\limits_{(v_3, v_4, v_{13}) \in \mc S_{a,b}} \vb{X_{a,b}^{v_3, v_4, v_{13}} (n)} S_{w_0} \rb{\theta_{a,b}^{v_3, v_4, v_{13}}; s},
\ea
where
\ba
\theta_{a,b}^{v_3, v_4, v_{13}} (\lambda,\lambda') = \e\rb{\frac{\rb{m_1\hat v_2 p^{s-r}}\lambda_1 + \rb{m_2\hat v_{14}} \lambda_2 + \rb{n_1p^{s-a}}\lambda'_1 + \rb{n_2p^{s-b}}\lambda'_2}{p^s}}. 
\ea
By \eqref{eq:Kloosterman_GL2decomp}, we have
\ba
S_{w_0} \rb{\theta_{a,b}^{v_3, v_4, v_{13}}; s} = S\rb{m_1\hat v_2 p^{s-r}, n_1p^{s-a}; p^s} S\rb{m_2\hat v_{14}, n_2p^{s-b}; p^s}. 
\ea
Applying \eqref{eq:Kloosterman_GL2bound} gives
\ba
\vb{S_{w_0} \rb{\theta_{a,b}^{v_3, v_4, v_{13}}; s}} \leq 4 p^s \rb{\gcd\rb{m_1\hat v_2p^{s-r}, n_1p^{s-a}, p^s}, \gcd\rb{m_2\hat v_{14}, n_2p^{s-b}, p^s}}^{1/2}.
\ea
\ee

Now we give a bound to the size of $\Kl_p\rb{n, \psi, \psi'}$. To ease computation, we consider a relaxed bound by ignoring $\hat v_2$ and $\hat v_{14}$.

Suppose $r>s$. Then the bound says
\ba
\vb{S_{w_0}\rb{\theta_{a,b}^{v_3, v_4, v_{13}}; r}} &\leq 4 p^r \rb{\gcd\rb{m_1\hat v_2, n_1p^{r-a}, p^r} \gcd\rb{m_2\hat v_{14} p^{r-s}, n_2p^{r-b}, p^r}}^{1/2}\\
&\leq 4 p^r\rb{\vb{n_1n_2}_p^{-1} p^{2r-a-b}}^{1/2}\\
&= 4 p^{2r-\frac{a+b}{2}} \vb{n_1n_2}_p^{-1/2}. 
\ea

Note that
\ba
\sum\limits_{\rb{v_3, v_4, v_{13}} \in \mc S_{a,b}} \vb{X_{a,b}^{v_3, v_4, v_{13}}(n)} \leq \vb{\mc S_{a,b}} p^{a+b}.
\ea

Hence
\ba
\vb{\Kl_p\rb{n, \psi, \psi'}} &\leq \sum\limits_{\substack{a\leq r\\b\leq s}} \vb{S_{a,b} \rb{n,\psi, \psi'}}\\
&\leq \sum\limits_{\substack{a\leq r\\ b\leq s}} p^{-2r}\rb{1-p^{-1}}^{-2} 4\vb{n_1n_2}_p^{-1/2} \vb{\mc S_{a,b}} p^{2r+\frac{a+b}{2}}\\
&\ll \vb{n_1n_2}_p^{-1/2} \sum\limits_{\substack{a\leq r\\ b\leq s}} \vb{\mc S_{a,b}} p^{\frac{a+b}{2}}.
\ea

So it suffices to give an upper bound to $\vb{\mc S_{a,b}}$. Such bounds were computed in \cite[Section 5]{Man2020}. Note that we require $r\geq a+b$ in order to have $\mc S_{a,b}$ nonempty. 

Case I: Suppose $s-r+a\geq 0$.
\bea
\item If $s-2r+2a+b \geq 0$, then $\vb{\mc S_{a,b}}\leq p^{r+s-a-b}$.
\item If $s-2r+2a+b < 0$, then $\vb{\mc S_{a,b}}\leq p^{2s-b-\lceil\frac{s-b}{2}\rceil} \leq p^{3s/2-b/2}$.
\ee

Case II: Suppose $s-r+a<0$. Then $\vb{\mc S_{a,b}}\leq p^{2s-b-\lceil\frac{s-b}{2}\rceil} \leq p^{3s/2-b/2}$.

Combining the cases, we obtain
\ba
\sum\limits_{\substack{a\leq r\\ b\leq s}} \vb{\mc S_{a,b}} p^{\frac{a+b}{2}} &\leq \sum\limits_{\substack{r-s\leq a\leq r\\ 2r-2a-s\leq b\leq r-a}} p^{r+s-\frac{a}{2}-\frac{b}{2}} + \sum\limits_{\substack{r-s\leq a\leq r\\ b<2r-2a-s}} p^{\frac{3s}{2}+\frac{a}{2}} + \sum\limits_{\substack{a<r-s\\ b\leq s}} p^{\frac{3s}{2}+\frac{a}{2}}\\
&\ll \rb{s+1}p^{\frac{r}{2}+\frac{5s}{4}}.\\
\ea
Hence, we have for $r>s$
\begin{align}\label{eq:w0Kloosterman_rgs}
\vb{\Kl_p\rb{n, \psi, \psi'}} \ll \vb{n_1n_2}_p^{-1/2} \rb{s+1}p^{\frac{r}{2}+\frac{5s}{4}}.
\end{align}

For $r\leq s$, applying the same argument gives
\begin{align}\label{eq:w0Kloosterman_rls}
\vb{\Kl_p\rb{n, \psi, \psi'}} \ll \vb{n_1n_2}_p^{-1/2} \rb{s-r+1} p^{r+\frac{3s}{4}}.
\end{align}

Combining \eqref{eq:w0Kloosterman_rgs} and \eqref{eq:w0Kloosterman_rls}, we get
\begin{align}\label{eq:w0Kloosterman_nn}
\vb{\Kl_p\rb{n, \psi, \psi'}} \ll \vb{n_1n_2}_p^{-1/2} \rb{s+1} p^{\frac{r}{2} + \frac{3s}{4} + \frac{1}{2}\min\cb{r,s}}.
\end{align}
By \Cref{prp:w0Kloosterman_swap}, we can swap the characters, so
\begin{align}\label{eq:w0Kloosterman_mm}
\vb{\Kl_p\rb{n, \psi, \psi'}} \ll \vb{m_1m_2}_p^{-1/2} \rb{s+1} p^{\frac{r}{2} + \frac{3s}{4} + \frac{1}{2}\min\cb{r,s}}
\end{align}
as well. Combining \eqref{eq:w0Kloosterman_nn} and \eqref{eq:w0Kloosterman_mm} yields the bound for $\Kl_p(n_{w_0,r,s},\psi,\psi')$.

\subsection{Bounds for global Kloosterman sums} By combining the bounds for local Kloosterman sums $\Kl_p(n_{w,r,s}, \psi, \psi')$, we obtain bounds for global Kloosterman sums, and prove \Cref{thm:global_bound}.

\begin{proof}[Proof of \Cref{thm:global_bound}]
The statement for $\Kl(n_{\id}(c_1,c_2),\psi,\psi')$ follows because 
\ba
\Kl_p(n_{\id,r,s},\psi,\psi') = \begin{cases} 1 & \text{if } r=s=0,\\ 0 & \text{otherwise.}\end{cases}
\ea 

Meanwhile, $\Kl(n_{s_\alpha}(c_1,1),\psi,\psi')$ and $\Kl(n_{s_\beta}(1,c_2),\psi,\psi')$ are just classical Kloosterman sums. Combining local bounds for classical Kloosterman sums gives the global bounds, which read
\ba
\vb{\Kl_p\rb{n_{s_\alpha}(c_1,1), \psi, \psi'}} &\ll_\varepsilon (m_1,n_1,c_1)^{1/2} c_1^{1/2+\varepsilon},\\
\vb{\Kl_p\rb{n_{s_\beta}(1,c_2), \psi, \psi'}} &\ll_\varepsilon (m_2, n_2, c_2)^{1/2} c_2^{1/2+\varepsilon}.
\ea

For $\Kl(n_{s_\alpha s_\beta}(c_1,c_2),\psi,\psi')$ and $\Kl(n_{s_\beta s_\alpha}(c_1,c_2),\psi,\psi')$, we again combine the local bounds given in \Cref{thm:local_bound} yields the global bounds.

For $\Kl(n_{s_\alpha s_\beta s_\alpha}(c_1,c_2),\psi,\psi')$ and $\Kl(n_{s_\beta s_\alpha s_\beta}(c_1,c_2),\psi,\psi')$, the situation is more complicated, since the shapes of the local bounds depend on the relative size of $r,s$. Therefore, in order to obtain a global bound, we have to find an expression for the local bound that works for all values of $r,s$.

We start with $\Kl(n_{s_\alpha s_\beta s_\alpha}(c_1,c_2),\psi,\psi')$. For $s\leq r$, we have
\ba
\scalebox{0.91}{$\displaystyle \vb{\Kl_p\rb{n_{s_\alpha s_\beta s_\alpha, r,s}, \psi, \psi'}} \ll p^{\frac{r}{3}+\frac{4s}{3}+\frac{2}{3}\min\cb{\ord_p(m_1),\ord_p(n_1)}+\frac{1}{3} \ord_p(m_2)} \leq p^{\frac{4r}{3}+\frac{s}{3}+\frac{2}{3}\min\cb{\ord_p(m_1),\ord_p(n_1)}+\frac{1}{3} \ord_p(m_2)}.$}
\ea
For $r<s<2r$, we have
\ba
\vb{\Kl_p\rb{n_{s_\alpha s_\beta s_\alpha, r,s}, \psi, \psi'}} \ll p^{r+\ord_p(m_2)} + p^{r+\frac{s}{2}+\min\cb{\ord_p(m_1),\ord_p(n_1)}},
\ea
and we have inequalities
\ba
p^{r+\ord_p(m_2)} + p^{r+\frac{s}{2}+\min\cb{\ord_p(m_1),\ord_p(n_1)}} &\leq p^{r+\ord_p(m_2)} + p^{\frac{4r}{3}+\frac{s}{3}+\min\cb{\ord_p(m_1),\ord_p(n_1)}},\\
p^{r+\ord_p(m_2)} + p^{r+\frac{s}{2}+\min\cb{\ord_p(m_1),\ord_p(n_1)}} &\leq p^{s+\ord_p(m_2)} + p^{\frac{r}{6}+\frac{4s}{3}+\min\cb{\ord_p(m_1),\ord_p(n_1)}}.
\ea
For $s=2r$, we have
\ba
\vb{\Kl_p\rb{n_{s_\alpha s_\beta s_\alpha, r,s}, \psi, \psi'}} \ll p^{r+\ord_p(m_2)} = p^{\frac{s}{2}+\ord_p(m_2)}.
\ea
So we can conclude for $0\leq s \leq 2r$ that
\ba
\vb{\Kl_p\rb{n_{s_\alpha s_\beta s_\alpha, r,s}, \psi, \psi'}} \ll p^{\min\cb{\frac{4r}{3}+\frac{s}{3}, \frac{r}{3}+\frac{4s}{3}}+\ord_p(m_2)+\min\cb{\ord_p(m_1),\ord_p(n_1)}}.
\ea
Since we may assume from \eqref{eq:abaKloosterman_theta_sum} that $\ord_p(m_1), \ord_p(n_1)\leq r$, and $\ord_p(m_2) \leq s$, we have
\ba
\vb{\Kl\rb{n_{s_\alpha s_\beta s_\alpha}(c_1,c_2), \psi,\psi'}} \ll_\varepsilon (m_1,n_1,c_1) (m_2,c_2) (c_1,c_2) (c_1c_2)^{1/3+\varepsilon}
\ea
for every $\varepsilon>0$.

Now we consider $\Kl(n_{s_\beta s_\alpha s_\beta}(c_1,c_2),\psi,\psi')$. For $r\leq s/2$, we have
\ba
\scalebox{0.91}{$\displaystyle \vb{\Kl_p\rb{n_{s_\beta s_\alpha s_\beta,r,s}, \psi, \psi'}} \ll p^{\frac{3r}{2}+\frac{s}{2}+\frac{1}{2}\ord_p(m_1)+\frac{1}{2}\min\cb{\ord_p(m_2),\ord_p(n_2)}} \leq p^{-\frac{r}{2}+\frac{3s}{2}+\frac{1}{2}\ord_p(m_1)+\frac{1}{2}\min\cb{\ord_p(m_2),\ord_p(n_2)}}.$}
\ea
For $s/2<r<s$, we have
\ba
\scalebox{0.91}{$\displaystyle \vb{\Kl_p\rb{n_{s_\beta s_\alpha s_\beta,r,s}, \psi, \psi'}} \ll p^{-\frac{r}{2}+\frac{3s}{2}+\frac{1}{2}\ord_p(m_1)+\frac{1}{2}\min\cb{\ord_p(m_2),\ord_p(n_2)}} \leq p^{\frac{3r}{2}+\frac{s}{2}+\frac{1}{2}\ord_p(m_1)+\frac{1}{2}\min\cb{\ord_p(m_2),\ord_p(n_2)}}.$}
\ea
For $s=r$, we have
\ba
\vb{\Kl_p\rb{n_{s_\beta s_\alpha s_\beta,r,s}, \psi, \psi'}} \ll p^{s+\min\cb{\ord_p(m_1),\ord_p(n_2)}} = p^{r+\min\cb{\ord_p(m_1),\ord_p(n_2)}}.
\ea
So we can conclude for $0\leq r\leq s$ that
\ba
\vb{\Kl_p\rb{n_{s_\beta s_\alpha s_\beta,r,s}, \psi, \psi'}} \ll p^{\min\cb{\frac{3r}{2}+\frac{s}{2}, -\frac{r}{2}+\frac{3s}{2}}+\ord_p(m_1)+\frac{1}{2}\min\cb{\ord_p(m_2),\ord_p(n_2)}}.
\ea
Since we may assume from \eqref{eq:babKloosterman_theta_sum} that $\ord_p(m_1)\leq r$, and $\ord_p(m_2), \ord_p(n_2) \leq s$, we have
\ba
\vb{\Kl\rb{n_{s_\beta s_\alpha s_\beta}(c_1,c_2), \psi,\psi'}} \ll_\varepsilon (m_1,c_1) (m_2,n_2,c_2) (c_1^2,c_2) c_1^{-1/2} c_2^{1/2} (c_1c_2)^\varepsilon
\ea
for every $\varepsilon>0$.

For $\Kl(n_{w_0}(c_1,c_2),\psi,\psi')$, the local bound again consists of a single expression, so the local bounds given in \Cref{thm:local_bound} can be combined directly to give the stated global bound.
\end{proof}

\section{Symplectic Poincaré series} \label{section:sym_Poincare}

In this section, we compute the Fourier coefficients of symplectic Poincaré series, in terms of auxiliary Kloosterman sums.
\begin{dfn}
\bea
\item
Let $n\in N\rb{\Q_p}$, and $\psi_p, \psi'_p$ be characters of $U\rb{\Q_p}$ which are trivial on $U\rb{\Z_p}$. Then the local auxiliary Kloosterman sum is defined to be
\ba
\Klu_p\rb{n, \psi_p, \psi'_p} = \sum\limits_{\substack{x\in X(n)\\ x = b_1 n b_2}} \psi_p\rb{b_1} \psi'_p\rb{b_2}
\ea
if $\psi_p\rb{nun^{-1}} = \psi'_p\rb{u}$ for $u \in \ol U_n\rb{\Q_p}$, and zero otherwise. We say $\Klu_p\rb{n, \psi_p, \psi'_p}$ is well-defined if $\psi_p\rb{nun^{-1}} = \psi'_p\rb{u}$ for $u \in \ol U_n\rb{\Q_p}$. 
\item
Let $n\in N\rb{\Q}$, and $\psi = \prod\limits_p \psi_p$, $\psi' = \prod\limits_p \psi'_p$ be characters of $U\rb{\A}$ which are trivial on $\prod\limits_p U\rb{\Z_p}$. Then the global auxiliary Kloosterman sum is defined to be
\ba
\Klu\rb{n, \psi, \psi'} = \prod\limits_p \Klu_p\rb{n, \psi_p, \psi'_p}. 
\ea
\ee
\end{dfn}

We first show that the auxiliary Kloosterman sums are well-defined.
\begin{prp}\label{prp:Friedberg1.3}
\cite[Proposition 1.3]{Friedberg1987} Let $G = \Sp\rb{2r, \Q_p}$, $n \in N\rb{\Q_p}$, and $x \in X(n)$, with Bruhat decomposition $x = b_1 n b_2$, with $b_1, b_2 \in U\rb{\Q_p}$. Let $\psi, \psi'$ be characters of $U\rb{\Q_p}$ which are trivial on $U\rb{\Z_p}$. Then the quantity $\psi\rb{b_1} \psi'\rb{b_2}$ is well-defined as a function on $X(n)$ if $\psi\rb{nun^{-1}} = \psi'\rb{u}$ for $u \in \ol U_n\rb{\Q_p}$. 
\end{prp}
\begin{proof}
Suppose $\psi\rb{nun^{-1}} = \psi'\rb{u}$ for all $u \in \ol U_n\rb{\Q_p}$. Let $x = b_1 n b_2 = b'_1 n b'_2$ be two Bruhat decompositions. This says $b'_1 = \gamma b_1$ for some $\gamma \in U(\Z_p)$, and $b'_2 = b_2 \delta$ for some $\delta \in U_n\rb{\Z_p}$. Then we have
\ba
U\rb{\Z_p} b_1 n b_2 \delta^{-1} = U\rb{\Z_p} b_1 n b_2,
\ea
which implies $b_2 {b'_2}^{-1} = b_2 \delta^{-1} b_2^{-1} \in \ol U_n\rb{\Q_p}$. Now, from the equivalence of Bruhat decompositions, we deduce that
\ba
U\rb{\Z_p} n b_2 {b'_2}^{-1} n^{-1} U_n\rb{\Z_p} = U\rb{\Z_p} b_1^{-1} b'_1 U_n\rb{\Z_p},
\ea
which implies $\psi' \rb{b_2{b'_2}^{-1}} = \psi \rb{nb_2{b'_2}^{-1} n^{-1}} = \psi\rb{b_1^{-1} b'_1}$. 
\end{proof}

\begin{prp}\label{prp:auxKl}
If $\Klu_p\rb{n, \psi_p, \psi'_p}$ is well-defined, then $\Klu_p\rb{n, \psi_p, \psi'_p} = \Kl_p\rb{n, \psi_p, \psi'_p}$.
\end{prp}
\begin{proof}
Trivial.
\end{proof}

The Fourier coefficients $P_{\psi, \psi'} (g)$ can be evaluated using the following theorem of Friedberg: 
\begin{thm}\label{thm:FriedbergThmA}
\cite[Theorem A]{Friedberg1987} The Fourier coefficient $P_{\psi, \psi'} (g)$ of $\Sp(2r)$ Poincaré series is given by
\ba
P_{\psi, \psi'} (g) = \sum\limits_{w\in W} \sum\limits_{\substack{n \in N\rb{\Q}\\ w(n) = w}} \Klu\rb{n, \psi, \psi'} \int_{U_w\rb{\R}} \mc F_\psi\rb{n u_1 y} \ol{\psi'} \rb{u_1} du_1.
\ea
\end{thm}

\begin{rmk}
In \cite{Friedberg1987}, the statement concerns $\GL(r)$ Poincaré series, but the proof also works for $\Sp(2r)$ Poincaré series.
\end{rmk}

\subsection{$\Sp(4)$ Poincaré series}

Let $G = \Sp\rb{4, \Q_p}$, and $\psi = \psi_{m_1, m_2}$, $\psi' = \psi_{n_1, n_2}$. We give a table of conditions for auxiliary $\Sp(4)$ Kloosterman sums $\Klu_p\rb{n_{w, r, s}, \psi, \psi'}$ to be well-defined.
\ba
\begin{array}{|c|c|c|c|}
\hline
w & {\text{Well-definedness conditions}} & w & {\text{Well-definedness conditions}}\\
\hline
\id & m_1=n_1, m_2=n_2 & s_\beta s_\alpha & m_1=n_2=0\\
\hline
s_\alpha  & m_2=n_2=0 & s_\alpha s_\beta s_\alpha & n_2 = m_2 p^{2r-2s}\\
\hline
s_\beta & m_1=n_1=0 & s_\beta s_\alpha s_\beta & n_1 = m_1 p^{s-2r}\\
\hline
s_\alpha s_\beta & m_2=n_1=0 & w_0 & -\\
\hline
\end{array}
\ea

\begin{rmk}
From this table, we see that not all Kloosterman sums $\Kl_p\rb{n, \psi, \psi'}$ correspond to a well-defined auxiliary Kloosterman sum $\Klu_p\rb{n, \psi, \psi'}$.
\end{rmk}

From the well-definedness conditions for $\Klu_p(n_{w,r,s},\psi,\psi')$, we see that when $\psi = \psi_{m_1,m_2}$, $\psi' = \psi_{n_1,n_2}$ are non-degenerate, i.e. $m_1m_2,n_1n_2\neq 0$, then 
\ba
w = \id, s_\alpha s_\beta s_\alpha, s_\beta s_\alpha s_\beta, w_0 \in W
\ea
are the only Weyl elements that contribute to the Fourier coefficient $P_{\psi,\psi'}(g)$.


\end{document}